\documentclass[a4paper]{article}

\usepackage[english]{babel}
\usepackage[utf8]{inputenc}
\usepackage{amsmath}
\usepackage{graphicx}
\usepackage[colorinlistoftodos]{todonotes}

\usepackage{amsmath,amssymb,latexsym,soul,cite,mathrsfs,amsthm}

\usepackage{color,enumitem,graphicx}

\usepackage[colorinlistoftodos]{todonotes}
\makeatletter
\providecommand\@dotsep{5}
\def\listtodoname{List of Todos}
\def\listoftodos{\@starttoc{tdo}\listtodoname}
\makeatother

\usepackage{verbatim}  

\usepackage{graphicx}            
\usepackage[all,2cell]{xy}       
\xyoption{2cell}
\usepackage[makeroom]{cancel}    

\usetikzlibrary{shapes}          

\newcommand{\LA}{\ensuremath{\mathcal{LA}}} 
\newcommand{\hh}{\mathfrak{h}}              
\renewcommand{\gg}{\mathfrak{g}}            
\newcommand{\ggl}{\mathfrak{gl}}            
\newcommand{\Lie}{\mathcal{L}}              
\DeclareMathOperator{\ad}{ad}               

\newcommand{\A}{\mathcal{A}}                
\newcommand{\Rr}{\mathbb R}					
\providecommand{\abs}[1]{\lvert#1\rvert}
\newtheorem{theorem}{Theorem}[section]
\newtheorem{lemma}[theorem]{Lemma}
\newtheorem{prop}[theorem]{Proposition}

\newtheorem{remark}[theorem]{Remark}
\theoremstyle{definition} \newtheorem{ex:}[theorem]{Example}
\theoremstyle{definition} \newtheorem{Def}[theorem]{Definition}
\newtheorem{Ack}[theorem]{Acknowledgements}

\newcommand{\grid}[1]{
\foreach \i [count=\row from 0, remember=\row as \lastrow (initially 0)] in {0,...,#1}{
    \foreach \j [count=\column from 0, remember=\column as \lastcolumn (initially 0)] in {0,...,\i}{
        \ifnum\row=0
            \node[tri](0-0){0-0};
        \else
            \ifnum\column=0
                \node[tri, anchor=north](\row-0) at (\lastrow-0.corner 2) {\row-0};
            \else
                \node[tri, anchor=north](\row-\column) at (\lastrow-\lastcolumn.corner 3) {\row-\column};
            \fi
        \fi}}
}

\title{A cohomology theory for Lie $2$-algebras}

\author{Camilo Angulo}

\date{\today}

\begin{document}
\maketitle

\begin{abstract}
In this article, we introduce a new cohomology theory associated to a Lie 2-algebras. This cohomology theory is shown to extend
the classical cohomology theory of Lie algebras; in particular, we show that the second cohomology group classify an appropriate type of extensions. 
\end{abstract}

\section{Introduction}
\label{sec:introduction}

Lie algebra-like structures have been paid much attention in the literature. In particular, categorifications of these have been considered in \cite{Lie2Gps,Lie2Alg,ZhuInt2Alg}. The objects we are going to study are known as strict Lie $2$-algebras in some parts of the literature. Examples of these are ideals of Lie algebras, classic representations of Lie algebras and, in the infinite dimensional case, the spaces of multiplicative vector fields of a Lie groupoid. In this sense, Lie $2$-algebras play a part in studying algebraic structures on quotients. 

In this article, we introduce a cohomology theory for Lie $2$-algebras. After going through some preliminaries, we define the notion of a representation of a Lie $2$-algebra with values on a $2$-vector space. We define a representation proceeding along the lines of the general philosophy: ``A representation of an object in a category $\mathcal{C}$ is a morphism to the space of endomorphisms of a flat abelian object in $\mathcal{C}$''.
We justify this approach by means of the following examples
\begin{itemize}
\item Vector spaces are flat abelian Lie groups, and representations of Lie groups are indeed maps to $GL(V)$.
\item Vector spaces are also abelian Lie algebras, and representations of Lie algebras are maps to $\ggl (V)$.
\item Vector bundles are abelian Lie groupoids whose $s$-fibres are flat, and a representation of a Lie groupoid is a map to $\xymatrix{GL(E) \ar@<0.5ex>[r] \ar@<-0.5ex>[r] & M}$.
\end{itemize}
Such spaces of endomorphisms might as well not be objects in the given category; therefore, to make sense of a representation when adopting this philosophy, one needs to ensure that the spaces of endomorphisms are indeed objects of $\mathcal{C}$.  We show that indeed this is the case for our representations and that these are the type of representations induced by an extension of a Lie $2$-algebra by a $2$-vector space. We define the complex pf Lie $2$-algebra cochains with values on these representations and prove that it is so.\\
In the last chapter we prove that the second cohomology group of the complex of Lie $2$-algebra cochains classifies extensions. We start out with an abstract extension and a splitting, and we write down the conditions for the naturally defined maps to build an extension back up. These ``cocycle equations'' are immediately given by the complex associated to the Lie $2$-algebras with values in a representation. Finally, we give an interpretation for the $0$th and first cohomology groups.
\begin{Ack}
The author would like to acknowledge the financial support of CAPES and CNPq, with which these research was carried out.
\end{Ack}

\section{Preliminaries}


We start by recalling the definition of a Lie $2$-algebra.
\begin{Def}
A \textit{Lie $2$-algebra} is a groupoid object internal to the category of Lie algebras. 
\end{Def}
We will write a generic Lie $2$-algebra as
\begin{eqnarray}\label{ALie2Alg}
\xymatrix{
\gg_1\times_\hh \gg_1 \ar[r]^{\quad \hat{m}} & \gg_1 \ar@<0.5ex>[r]^{\hat{s}} \ar@<-0.5ex>[r]_{\hat{t}} \ar@(l,d)[]_{\hat{\iota}} & \hh \ar[r]^{\hat{u}} & \gg_1 .
}
\end{eqnarray}

It has been noted in \cite{Lie2Gps,Lie2Alg,ZhuInt2Alg} and elsewhere that the category of Lie $2$-algebras is equivalent to the category of crossed modules of Lie algebras.
\begin{Def}
A \textit{crossed module of Lie algebras} is a Lie algebra morphism $\xymatrix{\gg \ar[r]^\mu & \hh}$ together with a Lie algebra action by derivations $\xymatrix{\Lie :\hh \ar[r] & \ggl (\gg )}$ satisfying
\begin{eqnarray*}
\mu(\Lie _y x) & = & [y, \mu (x)] ,\\
\Lie _{\mu(x_0)}x_1 & = & [x_0 ,x_1 ].
\end{eqnarray*} 
We will refer to these equations as equivariance and infinitesimal Peiffer respectively.
\end{Def} 
The equivalence between the category of Lie $2$-algebras and crossed modules of Lie algebras is given, at the level of objects, by the following. Associated to a crossed module $\xymatrix{\gg \ar[r]^\mu & \hh}$, the space of arrows of the Lie $2$-algebra is defined to be the semi-direct sum $\gg\oplus_\Lie\hh$, where the bracket is given by the usual formula
\begin{eqnarray*}
[(x_0 ,y_0 ),(x_1 ,y_1 )]_\Lie:=([x_0 ,x_1 ]+\Lie _{y_0}x_1 -\Lie _{y_1}x_0 ,[y_0 ,y_1 ]).
\end{eqnarray*}
The structural maps are given by 
\begin{align*}
\hat{s}(x,y)=y & \qquad\hat{t}(x,y)=y+\mu (x) & \hat{\iota}(x,y)=(-x,y+\mu (x)) & \qquad\hat{u}(y)=(0,y)            
\end{align*}
\begin{eqnarray*}
(x',y+\mu (x))\Join (x,y):=\hat{m}(x',y+\mu (x); x,y):=(x+x' ,y).
\end{eqnarray*}
Conversely, given a Lie $2$-algebra $\xymatrix{\gg_1 \ar@<0.5ex>[r] \ar@<-0.5ex>[r] & \hh}$, the associated crossed module is given by $\xymatrix{\ker\hat{s} \ar[r]^{\quad\hat{t}\vert_{\ker\hat{s}}} & \hh}$. The action for $\xi\in\ker \hat{s}$ and $y\in\hh$ is given by
\begin{eqnarray*}
\Lie _y\xi :=[\hat{u}(y),\xi]_1,
\end{eqnarray*} 
where $[,]_1$ is the bracket of $\gg_1$.\\
From here on out, we make no distinction between a Lie $2$-algebra and its associated crossed module.\\
\begin{ex:}
Lie algebras are Lie $2$-algebras whose associated crossed modules are \begin{eqnarray*}
\xymatrix{
(0) \ar[r] & \gg .
}
\end{eqnarray*}
\end{ex:}
\begin{ex:}
Ideals $\hh\trianglelefteq\gg$ are examples whose associated crossed modules are given by the inclusions 
\begin{eqnarray*}
\xymatrix{
\hh \ar[r] & \gg
}
\end{eqnarray*}
together with the adjoint action.
\end{ex:}
\begin{ex:}
classical representations of Lie algebras on a vector space $V$; the associated crossed modules being
\begin{eqnarray*}
\xymatrix{
V \ar[r]^0 & \gg ,
}
\end{eqnarray*}
where the actions are the ones given by the representations themselves.
\end{ex:}
We will see shortly that indeed these are morally all the examples, as any other will be a suitable combination of these. Before we move to that, a crucial infinite dimensional example.
\begin{ex:}
The Lie $2$-algebra of multiplicative vector fields on a Lie groupoid \cite{EugLer}
\end{ex:}
This latter example was noticed to be part of a larger phenomenon \cite{CristJames}; namely, that the category of functor-sections of an \LA -groupoid and their natural transformations always comes endowed with the structure of a Lie $2$-algebra. \\
   
Out of the crossed modules, one is able to read much of the groupoid theoretical data of a Lie $2$-algebra. Indeed, let $\gg_1$ be a Lie $2$-algebra with associated crossed module $\xymatrix{\gg \ar[r]^\mu & \hh}$ with action $\Lie$, then we have got the following collection of trivial observations.
\begin{lemma}
The orbit through $0$ is $\mu(\gg)$ and an ideal in $\hh$.
\end{lemma}
\begin{proof}
By definition, the orbit through $0$ is $\hat{t}(\hat{s}^{-1}(0))$. Now, $\hat{s}(x,y)=y$; thus, $\hat{s}^{-1}(y)=\gg\times\lbrace y\rbrace$. On the other hand, $\hat{t}(x,y)=y+\mu(x)$; therefore, $\hat{t}(\hat{s}^{-1}(0))=\hat{t}(\gg\times (0))=\mu(\gg)$ as claimed. As for the second part of the statement, the equivariance of $\mu$ yields, $[y,\mu(x)]=\mu(\Lie_yx)\in\mu(\gg)$.

\end{proof}
\begin{lemma}
The orbit through any point $y\in\hh$ is the coset of $\mu(\gg)$ in $\hh$ with respect to $y$; thus, every Lie $2$-algebra is regular and its leaf space is the Lie algebra $\hh/\mu(\gg)$.
\end{lemma}
\begin{remark}
One could be interested in Lie $2$-algebras precisely because they serve as models to study quotients. This is better seen in the case of their global counter-parts, Lie $2$-groups, when the topology comes into play. Simple examples as the inclusion of an irrational slope subgroup in the torus, whose quotient is simultaneously a group and a badly behaved topological space, can be studied using tools of differential geometry when looking at its Lie $2$-group instead. 
\end{remark} 
\begin{lemma}
The isotropy group of $\gg_1$ at $0$ is the central Lie subalgebra $\ker\mu$. All other isotropy groups for $y\in\hh$ are isomorphic.
\end{lemma}
\begin{proof}
First, notice that if $x_1,x_2\in\ker\mu$, by the infinitesimal Peiffer equation, $[x_1,x_2]=\Lie_{\mu(x_1)}x_2=0$; hence, $\ker\mu$ is indeed an abelian Lie algebra. By the same equation, taking $x_2$ an arbitrary element of $\gg$, it is central. Now, by definition the isotropy at $y$ is the intersection of its source and target fibres. Since, $\hat{t}^{-1}(y)=\lbrace(x,y')\in\gg\oplus\hh :y'+\mu(x)=y\rbrace$, $(\gg_1)_y=\lbrace (x,y)\in\gg\oplus\hh :y-\mu(x)=y\rbrace=\ker\mu\times\lbrace y\rbrace$ and the result follows.

\end{proof}
\begin{lemma}
The associated action of $\hh$ on the second isotropy is the honest representation defined by the restriction of $\Lie$. Moreover, there is an honest representation $\overline{\Lie}$ of the orbit space on $\ker\mu$ of which the previous representation is a pull-back.
\end{lemma}
\begin{proof}
We just need to prove that the representation is well-defined. Let $x\in\ker\mu$ and $y\in\hh$, then by equivariance, $\mu(\Lie_yx)=[y,\mu(x)]=0$. As for the second statement, define $\overline{\Lie}_{[y]}v=\Lie_yv$. It is well-defined too, as for any other representative $y'=y+\mu(x)$, $\Lie_{y'}v=\Lie_yv +\Lie_\mu(x)v$ and the infinitesimal Peiffer equation implies $\Lie_\mu(x)v=[x,v]=-[v,x]=-\Lie_{\mu(v)}x=0$. The rest follows trivially.

\end{proof}
These observations make clear that the data of a crossed module of Lie algebras can be reduced to a $4$-tuple $(\hh,I,V,\rho)$ of a Lie algebra $\hh$, an ideal $I\trianglelefteq\hh$, a vector space $V$ and a representation of $\hh/I$ on $V$. The crossed module of such a $4$-tuple is
\begin{eqnarray}
\xymatrix{
V\oplus I \ar[r] & \hh:(v,x) \ar@{|->}[r] & x,
}
\end{eqnarray}
where the Lie algebra structure on $V\oplus I$ is the direct product, and the action $\Lie_y(v,x)=(\rho_{[y]}v,[y,x])$. \\
We move now to an example that plays an important r\^ole in what follows.

\subsection{The linear Lie 2-algebra associated to a 2-vector space}
The structure of the linear Lie $2$-algebra that we are about to describe has already appeared in \cite{IntSubLin2} and \cite{Cristian&Olivier}. In the former reference, the linear Lie $2$-algebra is a special case of the DGLA of endomorphisms of a complex of vector spaces, and in the latter, it coincides with the gauge 2-groupoid of a complex of vector spaces over the point manifold. We introduce this structure from a different perspective. \\
A $2$-vector space is the same as an abelian Lie $2$-algebra, that is a groupoid object in the category of vector spaces. Using the equivalence between Lie $2$-algebras and crossed modules, one realizes that a $2$-vector space is equivalent to a $2$-term complex of vector spaces
\begin{eqnarray*}
\xymatrix{
W \ar[r]^\phi & V.
}
\end{eqnarray*}
We claim that the category of linear self functors of $\xymatrix{W\oplus V \ar@<0.5ex>[r] \ar@<-0.5ex>[r] & V}$ together with their linear natural transformations is a Lie $2$-algebra that we will call the linear Lie $2$-algebra and will note $\ggl (\phi)$.\\
First, linear functors correspond to pairs of linear maps $(F,f)$ commuting with $\phi$, i.e.
\begin{eqnarray*}
\xymatrix{
W \ar[r]^F \ar[d]_{\phi} & W \ar[d]^{\phi} \\
V \ar[r]_f             & V.
}
\end{eqnarray*}
A natural transformation $\alpha$ between two functors $(F_1,f_1)$ and $(F_2,f_2)$ is a linear map 
\begin{eqnarray*}
\xymatrix{
\alpha : V \ar[r] & W\oplus V : v \ar@{|->}[r] & (\alpha_1(v),\alpha_2(v)),
}
\end{eqnarray*}
such that the $s(\alpha_1(v),\alpha_2(v))=\alpha_2(v)=f_1(v)$ and $t(\alpha_1(v),\alpha_2(v))=\alpha_2(v)+\phi(\alpha_1(v))=f_2(v)$ and such that for every $(w,v)\in W\oplus V$ the diagram
\begin{eqnarray*}
\xymatrix{
f_1(v) \ar[rrr]^{(F_1(w),f_1(v)) \qquad} \ar[d]_{(\alpha_1(v),\alpha_2(v))} & & & f_1(v)+\phi(F_1(w)) \ar[d]^{(\alpha_1(v+\phi(w)),\alpha_2(v+\phi(w)))} \\
f_2(v) \ar[rrr]_{(F_2(w),f_2(v)) \qquad}                                    & & & f_2(v)+\phi(F_2(w)).
}
\end{eqnarray*}
commutes inside the category $W\oplus V$. That is
\begin{eqnarray*}
\alpha_1(v)+F_2(w)=\alpha_1(v+\phi(w))+F_1(w);
\end{eqnarray*}
which, by linearity of $\alpha$, is
\begin{eqnarray*}
\alpha_1(\phi(w))=F_2(w)-F_1(w).
\end{eqnarray*}
Summing up, the space of objects of the category of endomorphisms of $\xymatrix{W \ar[r]^\phi & V}$, $\ggl (\phi )$, is
\begin{eqnarray*}
\ggl (\phi )_0 =\lbrace (F,f) \in End(W)\oplus End(V):\phi\circ F=f\circ\phi \rbrace 
\end{eqnarray*} 
which is a vector space. Arrows between $(F_1,f_1)$ and $(F_2,f_2)$ are given by  
\begin{eqnarray*}
\ggl (\phi )((F_1,f_1);(F_2,f_2))=\lbrace A \in Hom(V,W):\phi\circ A=f_2 -f_1 ,\quad A\circ\phi =F_2 -F_1 \rbrace ;
\end{eqnarray*} 
therefore, 
\begin{eqnarray*}
\ggl (\phi ) = Hom(V,W)\oplus \ggl (\phi )_0 .
\end{eqnarray*} 
The structural maps are:
\begin{align*}
s(A;F,f) & =(F,f)   & t(A;F,f)     & =(F+A\phi ,f+\phi A)\\
u(F,f)   & =(0;F,f) & \iota(A;F,f) & =(-A;F+A\phi ,f+\phi A)
\end{align*}
\begin{eqnarray*}
(A';F+A\phi ,f+\phi A)\ast_v (A;F,f):=(A+A';F,f).
\end{eqnarray*}
Naturally, the multiplication is given by the so-called vertical composition of natural transformations. Notice that, incidentally, we found out that the category of endomorphisms of a $2$-vector space is not only a category, but a $2$-vector space itself. In the sequel, we will be referring to its associated $2$-term complex, 
\begin{eqnarray*}
\xymatrix{
Hom(V,W) \ar[r]^{\quad\Delta} & \ggl (\phi)_0,
}
\end{eqnarray*} 
where $\Delta A=(A\phi ,\phi A)$.\\
As we pointed out, there is an additional algebraic structure on $\ggl (\phi)$ that will endow it with a Lie $2$-algebra structure. Indeed, there is a bracket on the space of objects inherited from $\ggl (W)\oplus\ggl (V)$. For $(F_1,f_1),(F_2,f_2)\in \ggl (\phi )_0$, a simple computation yields $([F_1,F_2],[f_1,f_2])\in \ggl (\phi )_0$. Indeed,
\begin{align*}
\phi\circ [F_1,F_2] & = \phi(F_1F_2 -F_2F_1)     \\
                    & = \phi F_1F_2 -\phi F_2F_1 \\ 
                    & = f_1\phi F_2 -f_2\phi F_1 \\
                    & = f_1f_2\phi -f_2f_1\phi = [f_1,f_2]\circ\phi ,
\end{align*}
thus proving $\ggl (\phi)_0$ is a Lie subalgebra.\\
On the other hand, looking at the horizontal composition of natural transformations, one gets 
\begin{align*}
((A;F,f)\ast_h(B;G,g))_v & = (B(fv+\phi Av),g(fv+\phi Av))\Join (G(Av),g(fv)) \\
                         & = ((Bf+B\phi A+GA)v,gf(v)),
\end{align*}
where $(A;F,f),(B;G,g)\in \ggl (\phi )$. This product is bi-linear and associative; therefore, it can be used to define a Lie bracket on $\ggl (\phi )$, 
\begin{eqnarray*}
[(A_1;F_1,f_1),(A_2;F_2,f_2)]:=(A_2;F_2,f_2)\ast_h(A_1;F_1,f_1)-(A_1;F_1,f_1)\ast_h(A_2;F_2,f_2).
\end{eqnarray*}
We prove that $\ggl (\phi )$ endowed with this structure is a Lie $2$-algebr by looking at the associated $2$-term complex, which will thus inherit the structure of a crossed module of Lie algebras with bracket 
\begin{eqnarray*}
[A_1,A_2]_\phi = A_1\phi A_2-A_2\phi A_1
\end{eqnarray*}
on $Hom(V,W)$, and action 
\begin{eqnarray*}
\Lie^{\phi}_{(F,f)}A = FA-Af 
\end{eqnarray*}
for $(F,f)\in \ggl (\phi )_0$ and $A\in\ggl(\phi)_1$. We write $\ggl(\phi)_1$ instead of $Hom(V,W)$ to emphasize the presence of the bracket.  
\begin{prop}
Along with this bracket and this action,
\begin{eqnarray*}
\xymatrix{
\ggl(\phi)_1 \ar[r]^{\Delta} & \ggl(\phi)_0
}
\end{eqnarray*} 
is a crossed module of Lie algebras.
\end{prop}
\begin{proof}
This amounts to a routine check.
\begin{itemize}
\item $\Delta$ is a Lie algebra homomorphism: It is clearly linear and
\begin{align*}
\Delta [A_1,A_2]_\phi & = ([A_1,A_2]_\phi\phi ,\phi [A_1,A_2]_\phi) \\
                 & = ((A_1\phi A_2-A_2\phi A_1)\phi ,\phi (A_1\phi A_2-A_2\phi A_1)) \\
                 & = (A_1\phi A_2\phi -A_2\phi A_1\phi ,\phi A_1\phi A_2-\phi A_2\phi A_1) \\
                 & = ([A_1\phi, A_2\phi ],[\phi A_1\phi A_2])=[\Delta A_1,\Delta A_2].
\end{align*}
\item $\ggl(\phi)_0$ acts by derivations:
\begin{align*}
\Lie^\phi_{(F,f)}[A_1,A_2]_\phi & = F[A_1,A_2]_\phi-[A_1,A_2]_\phi f \\
                      & = F(A_1\phi A_2-A_2\phi A_1)-(A_1\phi A_2-A_2\phi A_1)f \\
                      & = FA_1\phi A_2-FA_2\phi A_1-A_1\phi A_2f+A_2\phi A_1f ,
\end{align*}
while on the other hand
\begin{align*}
[\Lie^\phi_{(F,f)}A_1,A_2]_\phi & = (\Lie^\phi_{(F,f)}A_1)\phi A_2 -A_2\phi\Lie^\phi_{(F,f)}A_1\\
                      & = FA_1\phi A_2-A_1f\phi A_2-A_2\phi FA_1+A_2\phi A_1f , 
\end{align*}
and
\begin{align*}
[A_1,\Lie^\phi_{(F,f)}A_2] & = A_1\phi\Lie^\phi_{(F,f)} A_2 -(\Lie^\phi_{(F,f)}A_2)\phi A_1\\
                      & = A_1\phi FA_2-A_1\phi A_2f-FA_2\phi A_1+A_2f\phi A_1 . 
\end{align*}
Hence, the desired equality follows by $\phi F=f\phi$.
\item $\xymatrix{\Lie^\phi :\ggl(\phi)_0 \ar[r] & \ggl (Hom(V,W))}$ is a Lie algebra homomorphism:
\begin{align*}
\Lie^\phi_{([F_1,F_2],[f_1,f_2])}A & = [F_1,F_2]A-A[f_1,f_2] \\
                              & = F_1F_2A-F_2F_1A-Af_1f_2+Af_2f_1 ,
\end{align*}
while on the other hand
\begin{align*}
\Lie^\phi_{(F_1,f_1)}\Lie^\phi_{(F_2,f_2)}A & = \Lie^\phi_{(F_1,f_1)}(F_2A-Af_2) \\
                                  & = F_1F_2A-F_1Af_2 -F_2Af_1+Af_2f_1 , 
\end{align*}
and
\begin{align*}
\Lie^\phi_{(F_2,f_2)}\Lie^\phi_{(F_1,f_1)}A & = \Lie^\phi_{(F_2,f_2)}(F_1A-Af_1) \\
                                  & = F_2F_1A-F_2Af_1 -F_1Af_2+Af_1f_2 . 
\end{align*}
Therefore, subtracting gives the desired identity.
\item $\Delta$ is equivariant:
\begin{align*}
\Delta(\Lie^\phi_{(F,f)}A) & = ((\Lie^\phi_{(F,f)}A)\phi ,\phi\Lie^\phi_{(F,f)}A) \\
                       & = (FA\phi -Af\phi ,\phi FA-\phi Af) \\
                       & = (FA\phi -A\phi F,f\phi A-\phi Af) \\
                       & = ([F,A\phi ],[f,\phi A])=[(F,f),\Delta A]
\end{align*}
\item Infinitesimal Peiffer: 
\begin{align*}
\Lie^\phi_{\Delta A_0}A_1 & = \Lie^\phi_{(A_0\phi ,\phi A_0)}A_1 \\
                     & = A_0\phi A_1-A_1\phi A_0 =[A_0,A_1]_\phi.
\end{align*}
\end{itemize}
\end{proof}

\subsection{The canonical simplicial object associated to a Lie $2$-algebra}
In this section, we point out that the cochain complex of a Lie $2$-algebra regarded as a Lie groupoid is compatible with the respective complexes that each space of the nerve has got. We will do so by regarding Lie algebras as special types of DG-algebras. \\

Let now $\gg$ be a Lie algebra and consider its Chevalley-Eilenberg complex $C^q(\gg)=\bigwedge^q \gg^*$, whose differential
\begin{eqnarray*}
\xymatrix{
\delta :C^\bullet (\gg) \ar[r] & C ^{\bullet +1} (\gg)
}
\end{eqnarray*}
is defined by the usual formula
\begin{align*}
(\delta_\gg\omega)(X)& =\sum_{m<n}(-1)^{m+n}\omega([x_m,x_n],X(m,n)),
\end{align*}
for $\omega\in C^q(\gg)$ and $X=(x_0,...,x_q)\in\gg^{q+1}$. Here and for the rest of this text, we will use the convention that
\begin{eqnarray*}
X(j)=(x_0,...,x_{j-1},x_{j+1},...,x_{q}) & \textnormal{and} & X(m,n)=(x_0,...,x_{m-1},x_{m+1},...,x_{n-1},x_{n+1},...,x_{q}),
\end{eqnarray*}
as opposed to the usual $\hat{\cdot}$ notation. \\
It is known that the wedge product verifies a graded Leibniz type rule for this differential and thus $(C^\bullet(\gg),\wedge,\delta_\gg)$ is usually referred to as a differential graded algebra or DG-algebra for short. Now, suppose conversely that we are given a differential $d$ making $(C^\bullet(\gg),\wedge,d)$ into a DG-algebra. 
As it turns out, this structure endows $\gg$ with a Lie algebra structure, whose bracket is defined using the canonical isomorphism $\gg\cong (\gg^*)^*$. Let $\omega\in C^1(\gg)$, and $x_1,x_2\in\gg$ and set
\begin{eqnarray*}
\omega([x_1,x_2]_d):=-d\omega(x_1,x_2).
\end{eqnarray*}
Clearly, this formula exploits the formula for the differential of the Chevalley-Eilenberg complex; hence, if we call $\gg_d$ the Lie algebra thus defined, $\delta_{\gg_d}=d$ and conversely $\gg_{\delta_\gg}=\gg$. We summarize this in the following proposition.
\begin{prop}(see e.g.\cite{Vaintrob})
Given a vector space $\gg$, there is a one-to-one correspondence between Lie algebra structures on $\gg$ and DG-algebra structures on $(C^\bullet(\gg),\wedge)$. 
\end{prop}
This correspondence extends to an isomorphism of categories. 
\begin{prop}\label{Lie=DGA}
Let 
\begin{eqnarray*}
\xymatrix{
\phi: \gg \ar[r] & \hh
}
\end{eqnarray*} 
be a linear map relating the Lie algebras $\gg$ and $\hh$. Then, $\phi$ is a Lie algebra morphism if, and only if $\phi^*$ is a map of DG-algebras from the complex $(C^\bullet(\hh),\delta_{\hh})$ to the complex $(C^\bullet(\gg),\delta_{\gg})$.
\end{prop}
In fact, notice that due to the compatibility with the wedge product, in order to proof that $\phi^*$ is a map of complexes, it is enough to see that the pull-back commutes with the exterior derivatives in degrees $0$ and $1$. Next, consider a Lie $2$-algebra as the one in \ref{ALie2Alg}, we will write $\gg_2:=\gg_1 \times_{\hh} \gg$ for the space of composable arrows. We write the rest of the nerve of the groupoid, which is by definition a simplicial manifold internal to the category of Lie algebras, as the multiplication and the projections are Lie algebra homomorphisms and $\gg_p :=\gg _1^{(p)}$ is a Lie subalgebra of $\gg _1^p$ for each $p$ 
\begin{eqnarray*}
\xymatrix{
... \gg_3\ar@<1.5ex>[r]\ar@<0.5ex>[r]\ar@<-0.5ex>[r]\ar@<-1.5ex>[r] & \gg_2 \ar[r]\ar@<1.0ex>[r]\ar@<-1.0ex>[r] &  \gg_1 \ar@<0.5ex>[r]\ar@<-0.5ex>[r] & \hh.
}
\end{eqnarray*}
The face maps are defined by
\begin{eqnarray*}
\partial_k (x_0,...,x_p)=
  \begin{cases}
    (x_1,...,x_p)                       & \quad \text{if } k=0  \\
    (x_0,...,\hat{m}(x_{k-1},x_k),...,x_p)  & \quad \text{if } 0<k\leq p\\
    (x_0,...,x_{p-1})                   & \quad \text{if } k=p+1, 
  \end{cases}
\end{eqnarray*}
for a given element $(x_0,...,x_p)\in\gg^{(p+1)}$.

By means of proposition \ref{Lie=DGA}, we get 
\begin{eqnarray*}
\xymatrix{
... C^\bullet(\gg_3) & \ar@<1.5ex>[l]\ar@<0.5ex>[l]\ar@<-0.5ex>[l]\ar@<-1.5ex>[l] C^\bullet(\gg_2) & \ar[l]\ar@<1.0ex>[l]\ar@<-1.0ex>[l] C^\bullet (\gg_1) & \ar@<0.5ex>[l]\ar@<-0.5ex>[l] C^\bullet (\hh).
}
\end{eqnarray*}
Using the usual combinatorics $\partial:= \sum_{k=0}^{p+1}(-1)^k \partial_k^*$, one builds the double complex
\begin{eqnarray}\label{doubleTriv}
\xymatrix{
\vdots                                             & \vdots                                            & 
\vdots                                                &      \\ 
\bigwedge ^3\hh ^*\ar[r]^{\partial}\ar[u]          & \bigwedge ^3\gg _1^* \ar[r]^{\partial}\ar[u]        & 
\bigwedge ^3\gg _2^*\ar[r]\ar[u]                      & \dots\\
\hh ^*\wedge \hh ^* \ar[r]^{\partial}\ar[u]^{\delta_0} & \gg _1^*\wedge\gg _1^* \ar[r]^{\partial}\ar[u]^{\delta_1} & 
\gg _2^*\wedge \gg _2^* \ar[r]\ar[u]^{\delta_2}           & \dots\\
\hh ^* \ar[r]^{\partial}\ar[u]^{\delta_0}              & \gg _1^* \ar[r]^{\partial}\ar[u]^{\delta_1}             & 
\gg _2^* \ar[r]\ar[u]^{\delta_2}                          & \dots 
}
\end{eqnarray} 
Here, the $p$th column is the Chevalley-Eilenberg complex of the Lie algebra $\gg_p$ whose differential $\delta_i$ is a shorthand for $\delta_{\gg_i}$. In order, the $q$th row is the subcomplex of alternating multilinear groupoid cochains of the groupoid $\xymatrix{\gg _1^q \ar@<0.5ex>[r] \ar@<-0.5ex>[r] & \hh^q}$. 
We will call the total complex  
\begin{eqnarray*}
\Omega ^k_{tot}(\gg _1)=\bigoplus _{p+q=k}\bigwedge ^q\gg _p^*,
\end{eqnarray*}
with differential $d=\delta +(-1)^{q}\partial$ the \textit{double complex of} $\gg_1$, and its cohomology simply as $2$-cohomology.  
\begin{remark}
This construction is but a shadow of a more general correspondence between Lie algebroid structures on vector bundles and DG-algebra structures on their `exterior algebras of sections. When $\A$ is the tangent prolongation of a Lie groupoid, the double complex is the famous Bott-Shulman complex.
\end{remark} 


\section{Representations of Lie 2-algebras}
\begin{Def}
A \textit{representation} of a Lie $2$-algebra $\gg_1 =\xymatrix{\gg \ar[r]^\mu & \hh}$ on $\mathbb{V} =\xymatrix{W \ar[r]^\phi & V}$ is a morphism of Lie $2$-algebras
\begin{eqnarray*}
\xymatrix{
\rho :\gg_1 \ar[r] & \ggl (\phi ).
}
\end{eqnarray*}
\end{Def}
We will refer to a map as the one above as a $2$-representation. We would like to remark that, when we say a morphism of Lie $2$-algebras, we mean a na\"ive linear functor respecting the Lie algebra structures, as opposed to more general types of morphisms (e.g. bibundles of Lie groupoids, or equivalently maps of $2$-term $L_\infty$-algebras). \\
By definition, a representation of a Lie $2$-algebra has two ``commuting'' representations of $\hh$ on $W$ and on $V$ coming from the map of Lie algebras at the level of objects,  
\begin{eqnarray*}
\xymatrix{
\rho_0 :\hh \ar[r] & \ggl (\phi )_0\leq\ggl (W)\oplus\ggl (V). 
}
\end{eqnarray*}
Its components $\xymatrix{\rho_0^1 :\hh \ar[r] & \ggl (W)}$, $\xymatrix{\rho_0^0 :\hh \ar[r] & \ggl (V)}$ are representations fitting in the diagram
\begin{eqnarray*}
\xymatrix{
W \ar[d]_\phi \ar[r]^{\rho_0^1(y)} & W \ar[d]^\phi \\
V \ar[r]_{\rho_0^0(y)}             & V,
}
\end{eqnarray*} 
for each $y\in\hh$. On the other hand, at the level of arrows, one has got the map 
\begin{eqnarray*}
\xymatrix{
\rho_1 :\gg \ar[r] & \ggl(\phi)_1.
}
\end{eqnarray*}
This map is a Lie algebra homomorphism; hence,
\begin{eqnarray*}
\rho_1([x_0,x_1])=\rho_1(x_0)\phi\rho_1(x_1)-\rho_1(x_1)\phi\rho_1(x_0),
\end{eqnarray*}
and it also verifies the following equations for all $x\in\gg$: 
\begin{eqnarray*}
\rho^0_0(\mu(x))=\phi\rho_1(x), & \rho^1_0(\mu(x))=\rho_1(x)\phi ,
\end{eqnarray*}
and for all $y\in\hh$
\begin{eqnarray*}
\rho_1(\Lie_y x)=\rho_0^1(y)\rho_1(x)-\rho_1(x)\rho_0^0(y). 
\end{eqnarray*}
We now build an honest representation of $\gg _1 =\gg\oplus_\Lie\hh$ on $W\oplus V$.
\begin{prop}\label{honestAlgRep}
Given a representation $2$-representation $\xymatrix{\rho :\gg_1 \ar[r] & \ggl (\phi )}$, there is an honest representation 
\begin{eqnarray*}
\xymatrix{
\bar{\rho}:\gg\oplus_\Lie\hh \ar[r] & \ggl (W\oplus V):(x,y) \ar@{|->}[r] & {}
}
\begin{pmatrix}
    \rho_0^1(y+\mu(x)) & \rho_1(x) \\
    0                  & \rho_0^0(y) 
\end{pmatrix}
\end{eqnarray*}
\end{prop}
\begin{proof}
$\bar{\rho}$ is clearly linear and
\begin{align*}
\bar{\rho}([(x_0,y_0),(x_1,y_1)]_\Lie) & = \bar{\rho}([x_0,x_1]+\Lie _{y_0}x_1-\Lie_{y_1}x_0,[y_0,y_1]) \\
                                       & = \begin{pmatrix}
             \rho_0^1([y_0,y_1]+\mu([x_0,x_1]+\Lie _{y_0}x_1-\Lie_{y_1}x_0)) & \rho_1([x_0,x_1]+\Lie _{y_0}x_1-\Lie_{y_1}x_0) \\
                                             0                               & \rho_0^0([y_0,y_1]) 
                                           \end{pmatrix}.
\end{align*}
The maps appearing in the diagonal of this matrix will agree right away with those in the diagonal of $[\bar{\rho}(x_0,y_0),\bar{\rho}(x_1,y_1)]$. The map in the top corner to the right will coincide as well, as is shown by the following computation:
\begin{align*}
\rho_1([x_0,x_1]+\Lie _{y_0}x_1-\Lie_{y_1}x_0) & = \rho_1([x_0,x_1])+\rho_1(\Lie _{y_0}x_1)-\rho_1(\Lie_{y_1}x_0) \\
                                               & = [\rho_1(x_0),\rho_1(x_1)]+\Lie _{\rho_0(y_0)}\rho_1(x_1)-\Lie_{\rho_0(y_1)}\rho_1(x_0), 
\end{align*}
but the bracket is
\begin{align*}
[\rho_1(x_0),\rho_1(x_1)] & = \rho_1(x_0)\phi\rho_1(x_1)-\rho_1(x_1)\phi\rho_1(x_0) \\
                          & = \rho_0^1(\mu(x_0))\rho_1(x_1)-\rho_0^1(\mu(x_1))\rho_1(x_0);
\end{align*}
therefore, since $\Delta\circ\rho_1=\rho_0\circ\mu$ and the action verifies
\begin{align*}
\Lie _{\rho_0(y_0)}\rho_1(x_1) & = \rho_0^1(y_0)\rho_1(x_1)-\rho_1(x_1)\rho_0^0(y_0),
\end{align*}
we have got
\begin{align*}
\rho_1([x_0,x_1]+\Lie _{y_0}x_1-\Lie_{y_1}x_0) & = \rho_0^1(y_0+\mu(x_0))\rho_1(x_1)-\rho_0^1(y_1+\mu(x_1))\rho_1(x_0)+ \\
                                               & \qquad\rho_1(x_1)\rho_0^0(y_0)-\rho_1(x_0)\rho_0^0(y_1),                                                  
\end{align*}
which is the entry in the top corner of $[\bar{\rho}(x_0,y_0),\bar{\rho}(x_1,y_1)]$.

\end{proof}
The semi-direct sum of Lie algebras with respect to this representation,  $\gg _1 {}_{\bar{\rho}}\ltimes\mathbb{V}$, can be endowed with a structure that we call semi-direct product of Lie $2$-algebras and whose associated crossed module is
\begin{eqnarray*}
\xymatrix{
\gg {}_{\rho_0^1\circ\mu}\oplus W \ar[r]^{\mu\times\phi} & \hh {}_{\rho_0^0}\oplus V,
}
\end{eqnarray*}
with action
\begin{eqnarray*}
\Lie^\rho_{(y,v)}(x,w)=(\Lie_y x,\rho_0^1(y)w-\rho_1(x)v).
\end{eqnarray*}
We now provide some examples.
\begin{ex:}
Trivial representations. Of course, all of the defining relations get trivially satisfied if $\rho\equiv 0$.
\end{ex:}
\begin{ex:}\label{unitRep}
Usual Lie algebra representations can also be seen as examples of $2$-representations. Setting $W=\lbrace 0\rbrace$, a $2$-representation is equivalent to a single representation $\rho$ of $\hh$ on $V$, such that $\rho_{\mu(x)}\equiv 0$ for every $x\in\gg$. Ultimately, then, a $2$-representation on the unit $2$-vector space is simply a representation of $\hh/\mu(\gg)$. In particular, $2$-representations of a unit Lie $2$-algebra on a unit $2$-vector space are usual Lie algebra representations. \\
Curiously enough, assuming $V=\lbrace 0\rbrace$, one gets the same prescription, i.e. $2$-representation on a Lie group internal to vector spaces consist also of a single representation of the orbit space.
\end{ex:}
\begin{ex:}
Representations up to homotopy. The data defining a $2$-representation of the unit Lie $2$-algebra $\xymatrix{\hh \ar@<0.5ex>[r] \ar@<-0.5ex>[r] & \hh}$ coincides with a representation up to homotopy with zero curvature. In general, notice that the semi-direct product of Lie $2$-algebras $\gg_1{}_{\bar{\rho}}\ltimes\mathbb{V}$ is also VB-groupoid over $\gg_1$; hence, there is an associated representation up to homotopy (cf. \cite{CristMat}). However, this representation is oblivious of the fact that it comes from a $2$-representation. In other words, for any given $2$-representation, the associated representation up to homotopy is the same.
\end{ex:}
\begin{ex:}
The adjoint representation. We describe the maps defining the adjoint representation of a Lie $2$-algebra on itself.
\begin{eqnarray*}
\xymatrix{ad_1:\gg \ar[r] & \ggl(\mu)_1} & \ad_1(x)(u):=-\Lie_u x ,  \\
\xymatrix{ad_0^1:\hh \ar[r] & \ggl(\gg)}   & \ad_0^1(y)(v):=\Lie_y v , \\
\xymatrix{ad_0^0:\hh \ar[r] & \ggl(\hh)}   & \ad_0^0(y)(u):=[y,u].    
\end{eqnarray*}
\end{ex:}
We would like to point out that in contrast with general Lie groupoids and Lie algebroids, the latter example shows that Lie $2$-algebras admit an adjoint representation extending the classic one. We will define a cohomology theory of Lie $2$-algebras with values in these $2$-representations, but first we will prove that  
this notion of $2$-representations is the kind of action induced in an abstract extension.
\begin{Def}
An \textit{extension} of the Lie $2$-algebra $\xymatrix{\gg \ar[r]^{\mu} & \hh}$ by the $2$-vector space $\xymatrix{W \ar[r]^{\phi} & V}$ is a Lie $2$-algebra $\xymatrix{\mathfrak{e}_1 \ar[r]^{\epsilon} & \mathfrak{e}_0}$ that fits in 
\begin{eqnarray*}
\xymatrix{
0 \ar[r] & W \ar[d]_\phi\ar[r]^{j_1} & \mathfrak{e}_1 \ar[d]_\epsilon\ar[r]^{\pi_1} & \gg \ar[d]^\mu\ar[r] & 0  \\
0 \ar[r] & V \ar[r]_{j_0}            & \mathfrak{e}_0 \ar[r]_{\pi_0}                & \hh \ar[r]           & 0, 
}
\end{eqnarray*}
where the top and bottom rows are short exact sequences and the squares are maps of Lie $2$-algebras.
\end{Def}
\begin{prop}\label{Ind2Rep} 
Given a Lie $2$-algebra extension of $\xymatrix{\gg \ar[r]^{\mu} & \hh}$ by a $2$-vector space $\xymatrix{W \ar[r]^{\phi} & V}$,
\begin{eqnarray*}
\xymatrix{
0 \ar[r] & W \ar[d]_\phi\ar[r]^{j_1} & \mathfrak{e}_1 \ar[d]_\epsilon\ar[r]_{\pi_1} & \gg \ar[d]^\mu\ar[r]\ar@/_/[l]_{\sigma_1} & 0  \\
0 \ar[r] & V \ar[r]^{j_0}            & \mathfrak{e}_0 \ar[r]_{\pi_0}                & \hh \ar[r]\ar@/_/[l]_{\sigma_1}      & 0, 
}
\end{eqnarray*}
and a linear splitting $\sigma$, there is an induced $2$-representation $\xymatrix{\rho^{\epsilon}_{\sigma}:\gg_1 \ar[r] & \ggl (\phi)}$ given by
\begin{align*}
\rho^0_0(y)v     & =[\sigma_0(y),v]_{\mathfrak{e}_0}              & \rho^1_0(y)w & =\Lie^{\epsilon}_{\sigma_0(y)} w \\
\rho_1(x)v       & =-\Lie^{\epsilon}_v\sigma_1 (x),               &              & 
\end{align*}
for $y\in\hh$, $v\in V$, $w\in E$ and $x\in\gg$.
\end{prop}
The proof will consist of a series of computations that we will postpone to introduce several pieces of notation that will make easier its writing. \\
Given a $2$-extension as the one in the statement of proposition \ref{Ind2Rep} and adopting the convention that the injective maps are inclusions, we use the linear splitting to get the usual isomorphisms $\hh\oplus V\cong\mathfrak{e}_0$ and $\gg\oplus W\cong\mathfrak{e}_1$, given by
\begin{eqnarray*}
\xymatrix{
(z,a) \ar@{|->}[r] & a + \sigma_k (z)
}
\end{eqnarray*}
for $k=0$ and $k=1$ respectively. Their inverses are
\begin{eqnarray*}
\xymatrix{
e \ar@{|->}[r] & (\pi_k (e),e-\sigma_k (\pi_k (e))).
}
\end{eqnarray*}
In order to upgrade these isomorphisms of vector spaces to Lie algebra isomorphisms, we consider 
\begin{align*}
[v_0+\sigma_0(y_0),v_1+\sigma_0(y_1)]_{\mathfrak{e}_0} & =[v_0,v_1]+[\sigma_0(y_0),v_1]+[v_0,\sigma_0(y_1)]+[\sigma_0(y_0),\sigma_0(y_1)] \\
                                                       & = \rho^0_0(y_0)v_1 - \rho^0_0(y_1)v_0 +[\sigma_0(y_0),\sigma_0(y_1)].
\end{align*}
We use the inverse isomorphism to define the bracket on $\hh\oplus V$. Since $\rho^0_0(y_0)v_1 - \rho^0_0(y_1)v_0\in V$,
\begin{align*}
\pi_0(\rho^0_0(y_0)v_1 - \rho^0_0(y_1)v_0 +[\sigma_0(y_0),\sigma_0(y_1)]) & =\pi_0([\sigma_0(y_0),\sigma_0(y_1)]_{\mathfrak{e}_0}) \\
                                                                          & =[\pi_0\sigma_0(y_0),\pi_0\sigma_0(y_1)]_\hh =[y_0,y_1];
\end{align*}
thus, the bracket is
\begin{eqnarray*}
[(y_0,v_0),(y_1,v_1)]_0:=([y_0,y_1],\rho^0_0(y_0)v_1 - \rho^0_0(y_1)v_0 -\omega_0(y_0,y_1)),
\end{eqnarray*}
where $\omega_0(y_0,y_1)$ is shorthand for $\sigma_0([y_0,y_1])-[\sigma_0(y_0),\sigma_0(y_1)]$. This is nothing but the usual twisted semi-direct product $\hh {}_{\rho^0_0}\oplus^{\omega_0} V$ from the theory of Lie algebra extensions. Hence, if one supposes conversely, that the bracket was defined using an abstract $\omega_0\in (\hh^*\wedge\hh^*)\otimes V$, one is going to rediscover that in order for such bracket to satisfy the Jacobi identity, $\omega_0$ needs to be a $2$-cocycle in the Lie algebra cohomology of $\hh$ with values in $\rho^0_0$. We recall that the general formula for the differential of the complex of Lie algebra cochains of $\gg$ with values in a representation $\rho$ on the vector space $V$ 
\begin{eqnarray*}
\xymatrix{
\delta :\bigwedge^\bullet\gg^*\otimes V \ar[r] & \bigwedge^{\bullet +1}\gg^*\otimes V
}
\end{eqnarray*}
is 
\begin{align*}
(\delta\omega)(X)& =\sum_{j=0}^q(-1)^j\rho(x_j)\omega(X(j))+\sum_{m<n}(-1)^{m+n}\omega([x_m,x_n],X(m,n)),
\end{align*}
for $\omega\in\bigwedge^q\gg^*\otimes V$ and $X=(x_0,...,x_q)\in\gg^{q+1}$. \\
Following the same reasoning, one finds out that $\mathfrak{e}_1\cong\gg {}_{\rho^1_0\circ\mu}\oplus^{\omega_1} W$ as Lie algebras, with $\omega_1(x_0,x_1)=\sigma_1([x_0,x_1])-[\sigma_1(x_0),\sigma_1(x_1)]$; however, we are going to be able to waive the necessity of $\omega_1$ using the rest of the crossed module structure. \\
We now turn to the homomorphism $\epsilon$. Consider
\begin{align*}
\epsilon(w+\sigma_1(x)) & =\phi(w)+\epsilon(\sigma_1(x)),
\end{align*}
and use the inverse isomorphism to define the crossed module map. Since $\pi$ is a crossed module map,
\begin{align*}
\pi_0(\phi(w)+\epsilon(\sigma_1(x))) & =\pi_0(\phi(w))+\pi_0(\epsilon(\sigma_1(x)))) \\
                                     & =\mu(\pi_1\sigma_1(x)))=\mu(x);
\end{align*}
thus, the crossed module map is
\begin{eqnarray*}
\xymatrix{
(x,w) \ar@{|->}[r] & (\mu(x),\phi(w)+\varphi(x)),
}
\end{eqnarray*}
where $\varphi(x):=\epsilon(\sigma_1(x))-\sigma_0(\mu(x))$. We repeat this strategy one last time to get the action,
\begin{align*}
\Lie^\epsilon_{v+\sigma_0(y)}(w+\sigma_1(x)) & =\Lie_v w +\Lie_{\sigma_0(y)}w+\Lie_v\sigma_1(x)+\Lie_{\sigma_0(y)}\sigma_1(x) \\
                                             & =\rho^1_0(y)w -\rho_1(x)v +\Lie_{\sigma_0(y)}\sigma_1(x) ,
\end{align*}
Since $\pi$ is a crossed module map,
\begin{align*}
\pi_1(\rho^1_0(y)w -\rho_1(x)v +\Lie_{\sigma_0(y)}\sigma_1(x)) & =\pi_1(\Lie^\epsilon_{\sigma_0(y)}\sigma_1(x)) \\
                                                               & =\Lie^\mu_{\pi_0(\sigma_0(y))}\pi_1(\sigma_1(x))=\Lie_y x
\end{align*}
thus, the action is given by the equation
\begin{eqnarray*}
\Lie_{(y,v)}(x,w):=(\Lie _y x,\rho^1_0(y)w -\rho_1(x)v -\alpha(y;x))
\end{eqnarray*}
where $\alpha(y;x):=\sigma_1(\Lie _y x)-\Lie^\epsilon_{\sigma_0(y)}\sigma_1(x)$.\\
Using this data and the infinitesimal Peiffer equation for $\mathfrak{e}$, one readily sees that 
\begin{eqnarray*}
\omega_1(x_0,x_1)=\rho_1(x_1)(\varphi(x_0))+\alpha(\mu(x_0);x_1).
\end{eqnarray*}
\begin{proof}(of Proposition \ref{Ind2Rep})
We make the computations necessary to prove that $\rho^{\epsilon}_{\sigma}$ is a $2$-representation.
\begin{itemize}
\item Well-defined: We use the exactness of the sequences to see that the maps land where they are supposed to.
\begin{eqnarray*}
\pi_0([\sigma_0(y),v]_{\mathfrak{e}_0})=[\pi_0(\sigma_0(y)),\pi_0(v)]_\hh =[y,0]_\hh =0 & \Longrightarrow & \rho_0^0(y)v\in V, \\
\pi_1(\Lie^{\epsilon}_{\sigma_0(y)} w)=\Lie_{\pi_0(\sigma_0(y))}\pi_1(w) =\Lie_y 0   =0 & \Longrightarrow & \rho_0^1(y)w\in W, \\
\pi_1(-\Lie^{\epsilon}_v\sigma_1 (x))=-\Lie_{\pi_0(v)}\pi_1(\sigma_1 (x)) =-\Lie_0 x =0 & \Longrightarrow & \rho_1(x)v\in W. 
\end{eqnarray*}
Further, thus defined, $\rho_0^0(y)\circ\phi=\phi\circ\rho_0^1(y)$ for each $y\in\hh$. Indeed,
\begin{align*}
\rho_0^0(y)(\phi(w)) & =[\sigma_0(y),\phi(w)]_{\mathfrak{e}_0} =[\sigma_0(y),\epsilon(w)]_{\mathfrak{e}_0} \\
                     & =\epsilon(\Lie^{\epsilon}_{\sigma_0(y)} w) \\
                     & =\phi(\Lie^{\epsilon}_{\sigma_0(y)} w)=\phi(\rho_0^1(y)w);                   
\end{align*}
in these equations, we used that $\epsilon\circ j_1=j_0\circ\phi$, that $\Lie^{\epsilon}_{\sigma_0(y)}w\in W$ and the equivariance for the crossed module $\epsilon$. This proves that for all $y\in \hh$, $\rho_0(y):=(\rho_0^0(y),\rho_0^1(y))\in\ggl (\phi)_0$, as desired.
\item $\rho^0_0$ Lie algebra homomorphism: We write each side of the equation, and then justify why their difference is zero.
\begin{align*}
\rho^0_0([y_0,y_1]_\hh)v & =[\sigma_0([y_0,y_1]_\hh),v]_{\mathfrak{e}_0},
\end{align*}
and
\begin{align*}
[\rho^0_0(y_0),\rho^0_0(y_1)]v & =\rho^0_0(y_0)\rho^0_0(y_1)v-\rho_0^0(y_1)\rho_0^0(y_0)v \\
                               & =\rho^0_0(y_0)([\sigma_0(y_1),v]_{\mathfrak{e}_0})-\rho_0^0(y_1)([\sigma_0(y_0),v]_{\mathfrak{e}_0}) \\
                               & =[\sigma_0(y_0),[\sigma_0(y_1),v]_{\mathfrak{e}_0}]_{\mathfrak{e}_0}-[\sigma_0(y_1),[\sigma_0(y_0),v]_{\mathfrak{e}_0}]_{\mathfrak{e}_0} \\
                               & =[[\sigma_0(y_0),\sigma_0(y_1)]_{\mathfrak{e}_0},v]_{\mathfrak{e}_0},
\end{align*}                               
where the last equality is due to the Jacobi identity. Then, considering the difference, we get
\begin{align*}
(\rho_0^0([y_0,y_1]_\hh)-[\rho^0_0(y_0),\rho^0_0(y_1)])v & =[\omega_0(y_0,y_1),v]_{\mathfrak{e}_0},
\end{align*}
which is zero given that $\omega_0(y_0,y_1)\in V$.
\item $\rho_0^1$ Lie algebra homomorphism: We proceed with the same strategy.
\begin{align*}
\rho_0^1([y_0,y_1]_\hh)w & =\Lie^{\epsilon}_{\sigma_0([y_0,y_1]_\hh)} w,
\end{align*}
and
\begin{align*}
[\rho_0^1(y_0),\rho_0^1(y_1)]w & =\rho_0^1(y_0)\rho_0^1(y_1)w-\rho_0^1(y_1)\rho_0^1(y_0)w \\
                               & =\Lie^{\epsilon}_{\sigma_0(y_0)}\Lie^{\epsilon}_{\sigma_0(y_1)}w-\Lie^{\epsilon}_{\sigma_0(y_1)}\Lie^{\epsilon}_{\sigma_0(y_0)}w \\
                               & =[\Lie^{\epsilon}_{\sigma_0(y_0)},\Lie^{\epsilon}_{\sigma_0(y_1)}]w =\Lie^{\epsilon}_{[\sigma_0(y_0),\sigma_0(y_1)]_{\mathfrak{e}_0}}w,
\end{align*}                               
where the last line is because $\Lie^\epsilon$ is a Lie algebra action. Then, considering the difference, we get
\begin{align*}
(\rho_0^1([y_0,y_1]_\hh)-[\rho_0^1(y_0),\rho_0^1(y_1)])w & =\Lie^{\epsilon}_{\omega_0(y_0,y_1)} w,
\end{align*}
which is zero given that $\omega_0(y_0,y_1)\in V$.
\item $\rho_1$ Lie algebra homomorphism: We proceed with the same strategy.
\begin{align*}
\rho_1([x_0,x_1]_\gg)v & =-\Lie^{\epsilon}_v\sigma_1([x_0,x_1]_\gg) ,
\end{align*}
and
\begin{align*}
[\rho_1(x_0),\rho_1(x_1)]_\phi v & =\rho_1(x_0)\phi\rho_1(x_1)v-\rho_1(x_1)\phi\rho_1(x_0)v \\
                                 & =-\rho_1(x_0)\epsilon(\Lie^{\epsilon}_v\sigma_1(x_1))+\rho_1(x_1)\epsilon(\Lie^{\epsilon}_v\sigma_1(x_0)) \\
                                 & =-\rho_1(x_0)[v,\epsilon(\sigma_1(x_1))]_{\mathfrak{e}_0}+\rho_1(x_1)[v,\epsilon(\sigma_1(x_0))]_{\mathfrak{e}_0} \\
                                 & =\Lie^{\epsilon}_{[v,\epsilon(\sigma_1(x_1))]_{\mathfrak{e}_0}}\sigma_1(x_0)-\Lie^{\epsilon}_{[v,\epsilon(\sigma_1(x_0))]_{\mathfrak{e}_0}}\sigma_1(x_1). 
\end{align*}                               
Now, 
\begin{align*}
\Lie^{\epsilon}_{[v,\epsilon(\sigma_1(x_1))]_{\mathfrak{e}_0}}\sigma_1(x_0) & 
                                    =\Lie^{\epsilon}_v\Lie^{\epsilon}_{\epsilon(\sigma_1(x_1))}\sigma_1(x_0)-\Lie^{\epsilon}_{\epsilon(\sigma_1(x_1))}\Lie^{\epsilon}_v\sigma_1(x_0) \\
                                  & =\Lie^{\epsilon}_v[\sigma_1(x_1),\sigma_1(x_0)]_{\mathfrak{e}_1}-[\sigma_1(x_1),\Lie^{\epsilon}_v\sigma_1(x_0)]_{\mathfrak{e}_1},
\end{align*} 
thanks to the infinitesimal Peiffer equation, and consequently,
\begin{align*}
[\rho_1(x_0),\rho_1(x_1)]_\phi v & =\Lie^{\epsilon}_v[\sigma_1(x_1),\sigma_1(x_0)]_{\mathfrak{e}_1}-[\sigma_1(x_1),\Lie^{\epsilon}_v\sigma_1(x_0)]_{\mathfrak{e}_1}\\
                                 & \qquad -\Lie^{\epsilon}_v[\sigma_1(x_0),\sigma_1(x_1)]_{\mathfrak{e}_1}+[\sigma_1(x_0),\Lie^{\epsilon}_v\sigma_1(x_1)]_{\mathfrak{e}_1} \\
                                 & =-2\Lie^{\epsilon}_v[\sigma_1(x_0),\sigma_1(x_1)]_{\mathfrak{e}_1}+\Lie^{\epsilon}_v[\sigma_1(x_0),\sigma_1(x_1)]_{\mathfrak{e}_1} \\
                                 & =-\Lie^{\epsilon}_v[\sigma_1(x_0),\sigma_1(x_1)]_{\mathfrak{e}_1},
\end{align*} 
where the second equality follows since $\Lie^\epsilon$ is an action by derivations. Then, considering the difference, we get
\begin{align*}
(\rho_1([x_0,x_1]_\gg)-[\rho_1(x_0),\rho_1(x_1)])v & =-\Lie^{\epsilon}_v\omega_1(x_0,x_1),
\end{align*}
which is zero given that $\omega_1(x_0,x_1)\in W$.
\item $\rho_0\circ\mu =\Delta\circ\rho_1$: This equation breaks into two components, one in $\ggl (W)$ and one in $\ggl (V)$; namely, 
\begin{eqnarray*}
\rho_0^1(\mu(x))=\rho_1(x)\circ\phi , & \rho_0^0(\mu(x))=\phi\circ\rho_1(x) ,
\end{eqnarray*} 
for each $x\in\gg$. These relations follow as, using the strategy above,
\begin{align*}
\rho_0^1(\mu(x))w = \Lie^{\epsilon}_{\sigma_0(\mu(x))}w,\quad & \quad\rho_0^0(\mu(x))v = [\sigma_0(\mu(x)),v]_{\mathfrak{e}_0} ,
\end{align*}
and
\begin{align*}
\rho_1(x)\phi(w) & =-\Lie^{\epsilon}_{\phi(w)}\sigma_1(x)\quad & \quad\phi(\rho_1(x)v) & =\phi(-\Lie^{\epsilon}_v\sigma_1(x))    \\
                 & =-\Lie^{\epsilon}_{\epsilon(w)}\sigma_1(x) &                   & =-\epsilon(\Lie^{\epsilon}_v\sigma_1(x))     \\
                 & =-[w,\sigma_1(x)]_{\mathfrak{e}_1}&                            & =-[v,\epsilon(\sigma_1(x))]_{\mathfrak{e}_0} \\
                 & =[\sigma_1(x),w]_{\mathfrak{e}_1}&                             & =[\epsilon(\sigma_1(x)),v]_{\mathfrak{e}_0}. \\
                 & =\Lie^{\epsilon}_{\epsilon(\sigma_1(x))}w,
\end{align*}
Thus, considering the respective differences, we get
\begin{align*}
(\rho_1(x)\phi -\rho_0^1(\mu(x)))w & =\Lie^{\epsilon}_{\varphi(x)}w,\quad & \quad(\phi\rho_1(x)-\rho_0^0(\mu(x)))v & = [\varphi(x),v]_{\mathfrak{e}_0}
\end{align*}
which are both zero, since $\varphi(x)\in V$ as desired.
\item $\rho_1$ respects the actions: One last time.
\begin{align*}
\rho_1(\Lie_y x)v & =-\Lie^{\epsilon}_v\sigma_1(\Lie_y x),
\end{align*}
and
\begin{align*}
(\Lie^{\phi}_{\rho_0(y)}\rho_1(x))v & =\rho_0^1(y)\rho_1(x)v-\rho_1(x)\rho_0^0(y)v \\
                                    & =\rho_0^1(y)(-\Lie^{\epsilon}_v\sigma_1(x))-\rho_1(x)[\sigma_0(y),v]_{\mathfrak{e}_0} \\
                                    & =-\Lie^{\epsilon}_{\sigma_0(y)}\Lie^{\epsilon}_v\sigma_1(x)+\Lie^{\epsilon}_{[\sigma_0(y),v]_{\mathfrak{e}_0}}\sigma_1(x) \\
                                    & =-\Lie^{\epsilon}_v\Lie^{\epsilon}_{\sigma_0(y)}\sigma_1(x)
\end{align*}                               
where the last equality is again since $\Lie^\epsilon$ is a Lie algebra action. Then, considering the difference, we get
\begin{align*}
(\rho_1(\Lie_y x)-\Lie^{\phi}_{\rho_0(y)}\rho_1(x))v & =-\Lie^{\epsilon}_v\alpha(y;x),
\end{align*}
which is zero given that $\alpha(y;x)\in W$.
\end{itemize}
\end{proof} 
Notice that in the case where the splitting can be taken to be a crossed module morphism, and accordingly $\omega$, $\alpha$ and $\varphi$ vanish, the given formulas for the crossed module structure on the extension coincide with those of the semi-direct sum defined right after proposition \ref{honestAlgRep}.


\section{The complex of Lie 2-algebra cochains with values in a 2-representation}
Let $\gg_1:=\xymatrix{\gg \ar[r]^{\mu} & \hh}$ be a Lie $2$-algebra and $\rho$ be a $2$-representation on the $2$-vector space $\xymatrix{W \ar[r]^{\phi} & V}$ with components $\rho_0^0$, $\rho_0^1$, $\rho_1$, as before. In this section, we define the complex of $\gg_1$ cochains with values in $\rho$, $(C_{tot}(\gg_1,\phi),\nabla)$, and prove the main theorem.
\begin{theorem}\label{The2AlgCx}
$C_{tot}(\gg_1,\phi)$ graded by
\begin{eqnarray*}
C^{n}_{tot}(\gg_1,\phi)=\bigoplus_{p+q+r=n}C^{p,q}_r(\gg_1,\phi)
\end{eqnarray*}
together with the differential
\begin{eqnarray*}
\nabla=\delta^{(r)}+(-1)^q\delta_{(1)}+(-1)^{q+r}\partial+(-1)^r\sum_k\Delta_k
\end{eqnarray*}
is a complex.
\end{theorem}

From the grading by ``counter-diagonal planes'', one could anticipate that the complex comes as the total complex of a triple complex of sorts. In fact, as we will spell out shortly, the first three terms in the differential $\nabla$ are respectively the differentials of complexes in the $q$, $r$ and $p$-directions; however, as not all these differentials commute, one is forced to introduce the \textit{difference maps} $\Delta_k$. \\
We will assume the convention that $\gg_0 =\hh$. For $r\neq 0$,
\begin{eqnarray*}
C^{p,q}_r(\gg _1,\phi):=\bigwedge ^q\gg _p^*\otimes\bigwedge ^r\gg ^*\otimes W;
\end{eqnarray*}
whereas, 
\begin{eqnarray*}
C^{p,q}_0(\gg _1,\phi):=\bigwedge ^q\gg _p^*\otimes V.
\end{eqnarray*}
This three dimensional lattice of vector spaces comes together with a grid of maps that is a complex in each direction. We now proceed to describe them. \\
For constant $p$ and $r$, we have the complex
\begin{eqnarray*}
\xymatrix{
\delta^{(r)}:\bigwedge ^\bullet\gg _p^*\otimes\bigwedge ^r\gg ^*\otimes W \ar[r] & \bigwedge ^{\bullet+1}\gg _p^*\otimes\bigwedge ^r\gg ^*\otimes W
}
\end{eqnarray*}
in the $q$-direction. For $r=0$, this is the Chevalley-Eilenberg complex of $\gg_p$ with values on $V$. The induced representation $\rho_p$ of $\gg_p$ on $V$ is the pull-back of the honest representation $\rho_0^0$ along the ``final target'' map, which after identifying $\gg_p$ with $\gg^p\oplus\hh$, can be written as
\begin{eqnarray*}
\xymatrix{
\hat{t}_p:\gg_p \ar[r] & \hh :(x_1,...,x_p,y) \ar@{|->}[r] & y+\sum_{j=1}^p\mu(x_j).
}
\end{eqnarray*}
On the other hand, for $r\neq 0$, the complex is defined to be the Chevalley-Eilenberg complex of $\gg_p$ with values on $\bigwedge^r\gg^*\otimes W$. This time around, the representation is the pull-back along $\hat{t}_p$ of
\begin{eqnarray*}
\xymatrix{
\rho^{(r)}:\hh \ar[r] & \ggl (\bigwedge ^r\gg ^*\otimes W)
}
\end{eqnarray*}
given for $\omega\in\bigwedge ^r\gg ^*\otimes W$, $x_1,...,x_r\in\gg$ and $y\in\hh$ by
\begin{eqnarray*}
\rho ^{(r)}(y)\omega(x_1,...,x_r):=\rho_0^1(y)\omega(x_1,...,x_r)-\sum_{k=1}^r\omega(x_1,...,\Lie _y x_k,...,x_r).
\end{eqnarray*}
\begin{remark}
We could have alse taken the pull-back along the ``initial source'' map, which is the reversal $\hat{s}_p=\hat{t}_p\circ\hat{\iota}$. There is no economy in working with either representation: one will always pay a computational price somewhere. 
\end{remark}
\begin{remark}
The representation $\rho^{(r)}$ has already appeared in the literature \cite{HochSer}. The form in which they appear is restricted to the example of the inclusion of an ideal, regarded as a Lie $2$-algebra. 
\end{remark}
For constant $q$ and $r$, we have the complex
\begin{eqnarray*}
\xymatrix{
\partial:\bigwedge^q\gg_\bullet^*\otimes\bigwedge ^r\gg ^*\otimes W \ar[r] & \bigwedge^q\gg_{\bullet+1}^*\otimes\bigwedge ^r\gg ^*\otimes W
}
\end{eqnarray*}
in the $p$-direction. This is a subcomplex of the groupoid cochain complex of $\xymatrix{\gg _1^q \ar@<0.5ex>[r] \ar@<-0.5ex>[r] & \hh^q}$ whose differential is defined by the same formulae of the horizontal maps in diagram (\ref{doubleTriv}) above, but taking sums in $V$ for $r=0$ or in  $\bigwedge^r\gg^*\otimes W$ rather than in $\Rr$. \\
For constant $p$ and $q$, we have the complex
\begin{eqnarray*}
\xymatrix{
\delta_{(1)}:\bigwedge^q\gg_p^*\otimes\bigwedge^{\bullet}\gg ^*\otimes W \ar[r] & \bigwedge^q\gg_p^*\otimes\bigwedge ^{\bullet+1}\gg^*\otimes W
}
\end{eqnarray*}
in the $r$-direction. This is the Chevalley-Eilenberg complex of $\gg$ with values on $\bigwedge^q\gg_p^*\otimes W$, though having an unusual $0$th degree. The representation is an extension of $\rho_0^1\circ\mu$. Explicitely,
\begin{eqnarray*}
\xymatrix{
\rho_{(1)}:\gg \ar[r] & \ggl(\bigwedge^q\gg _p^*\otimes W)
}
\end{eqnarray*}
is given for $\alpha\in\bigwedge^q\gg _p^*\otimes W$, $\Xi\in\gg_p^q$ and $x\in\gg$ by
\begin{eqnarray*}
\rho _{(1)}(x)\alpha(\Xi):=\rho_0^1(\mu(x))\alpha(\Xi).
\end{eqnarray*}
Since the $0$th degree is $\bigwedge^q\gg_p^*\otimes V$ instead of $\bigwedge^q\gg_p^*\otimes W$, we define the first differential to be
\begin{eqnarray*}
\xymatrix{
\delta_{(1)}:\bigwedge^q\gg_p^*\otimes V \ar[r] & \bigwedge^q\gg_p^*\otimes\gg^*\otimes W
} \\
\delta_{(1)}\omega(\Xi;x)=\rho_1(x)\omega(\Xi)\qquad\qquad
\end{eqnarray*}
for $\Xi\in\gg_p^q$ and $x\in\gg$. We prove that in spite of this replacement, there is a complex in the $r$-direction.
\begin{lemma}\label{Alg r-cx}
The map 
\begin{eqnarray*}
\xymatrix{
\delta_{(1)}:\bigwedge^q\gg_p^*\otimes V \ar[r] & \bigwedge^q\gg_p^*\otimes\gg^*\otimes W
}
\end{eqnarray*}
defined above, fits in the complex
\begin{eqnarray*}
\xymatrix{
\bigwedge^q\gg_p^*\otimes V \ar[r]^{\delta_{(1)}\quad} & \bigwedge^q\gg_p^*\otimes\gg^*\otimes W \ar[r] & \bigwedge^q\gg_p^*\otimes\bigwedge^2\gg^*\otimes W \ar[r] & ...
}
\end{eqnarray*}
of $\gg$ with values in $\rho_{(1)}$
\end{lemma}
\begin{proof}
We prove that, for $\omega\bigwedge^q\gg_p^*\otimes V$, $\delta_{(1)}^2\omega=0$. Let $\Xi\in\gg_p^q$ and $x_0,x_1\in\gg$, then
\begin{align*}
\delta_{(1)}\delta_{(1)}\omega(\Xi;x_0,x_1) & =\rho_{(1)}(x_0)\delta_{(1)}\omega(\Xi;x_1)-\rho_{(1)}(x_1)\delta_{(1)}\omega(\Xi;x_0)-\delta_{(1)}\omega(\Xi;[x_0,x_1]) \\
                                            & =\rho_0^1(\mu(x_0))\rho_1(x_1)\omega(\Xi)-\rho_0^1(\mu(x_1))\rho_1(x_0)\omega(\Xi)-\rho_1([x_0,x_1])\omega(\Xi).
\end{align*} 
The result follows from $\rho_0^1(\mu(x_k))=\rho_1(x_k)\circ\phi$, since $\rho_1$ is a Lie algebra homomorphism landing in $\ggl(\phi)_1$.

\end{proof}
Finally, the $k$th difference map  
\begin{eqnarray*}
\xymatrix{
\Delta_{k}:C^{p,q}_r(\gg _1,\phi) \ar[r] & C^{p+1,q+k}_{r-k}(\gg _1,\phi)
}
\end{eqnarray*}
is defined by
\begin{eqnarray*}
\Delta_k\omega(\Xi;Z):=\sum_{a_1<...<a_k}(-1)^{a_1+...+a_k}\omega(\partial_0\Xi(a_1,...,a_k);x_{a_1}^0,...,x_{a_k}^0,Z)
\end{eqnarray*}
for $Z\in\gg^{r-k}$, $\Xi=(\xi_0,...,\xi_{q+k-1})\in\gg_{p+1}^{q+k}$ and $\xi_j$ is identified with $(x_j^0,...,x_j^p;y_j)$. Recall that by definition
\begin{eqnarray*}
\partial_k\Xi=(\partial_k\xi_0,...,\partial_k\xi_q);
\end{eqnarray*}
hence, trivially, $\partial_k(\Xi(a_1,...,a_k))=(\partial_k\Xi)(a_1,...,a_k)$, so one can drop the parenthesis. This shows that there is no ambiguity in the formula above. Notice further that all difference maps are homogeneous of degree $+1$ with respect to the diagonal grading; hence, it makes sense to use them as differentials. In the special case $r=k$, the map is essentially defined by the same formula, but composed with $\phi$, so that it takes values in the right vector space. We will sometimes drop the subindex for $k=1$. 
\begin{prop}\label{p-pagDblCx}
$(\delta^{(r)}\delta_{(1)}-\delta^{(r+1)}\delta_{(1)})\omega=0\in C^{p,q+1}_{r+1}(\gg_1,\phi)$; thus, for constant $p$, there is an actual double complex that we call the $p$-page.
\end{prop}
\begin{proof}
Let $\Xi=(\xi_0,...,\xi_q)\in\gg_{p}^{q+1}$ and $Z=(z_0,...,z_r)\in\gg^{r+1}$, then
\begin{align*}
\delta & ^{(r+1)}\delta_{(1)}\omega(\Xi;Z)=\sum_{j=0}^q(-1)^j\rho^{(r+1)}(\hat{t}_p(\xi_j))\delta_{(1)}\omega(\Xi(j);Z)+\sum_{m<n}(-1)^{m+n}\delta_{(1)}\omega([\xi_m,\xi_n],\Xi(m,n);Z) \\
             & =\sum_{j=0}^q(-1)^j\rho^{(r+1)}(\hat{t}_p(\xi_j))\Big{(}\sum_{k=0}^r(-1)^k\rho_0^1(\mu(z_k))\omega(\Xi(j);Z(k))+\sum_{a<b}(-1)^{a+b}\omega(\Xi(j);[z_a,z_b],Z(a,b))\Big{)}+ \\
             & \qquad +\sum_{m<n}(-1)^{m+n}\Big{(}\sum_{k=0}^r(-1)^k\rho_0^1(\mu(z_k))\omega([\xi_m,\xi_n],\Xi(m,n);Z(k))+ \\
             & \qquad\qquad +\sum_{a<b}(-1)^{a+b}\omega([\xi_m,\xi_n],\Xi(m,n);[z_a,z_b],Z(a,b))\Big{)} ;
\end{align*}
whereas, on the other hand,
\begin{align*}
\delta & _{(1)}\delta^{(r)}\omega(\Xi;Z)=\sum_{k=0}^r(-1)^k\rho_0^1(\mu(z_k))\delta^{(r)}\omega(\Xi;Z(k))+\sum_{a<b}(-1)^{a+b}\delta^{(r)}\omega(\Xi;[z_a,z_b],Z(a,b)) \\
             & =\sum_{k=0}^r(-1)^k\rho_0^1(\mu(z_k))\Big{(}\sum_{j=0}^q(-1)^j\rho^{(r)}(\hat{t}_p(\xi_j))\omega(\Xi(j);Z(k))+\sum_{m<n}(-1)^{m+n}\omega([\xi_m,\xi_n],\Xi(m,n);Z(k))\Big{)}+ \\
             & \qquad +\sum_{a<b}(-1)^{a+b}\Big{(}\sum_{j=0}^q(-1)^j\rho^{(r)}(\hat{t}_p(\xi_j))\omega(\Xi(j);[z_a,z_b],Z(a,b))+ \\
             & \qquad\qquad +\sum_{m<n}(-1)^{m+n}\omega([\xi_m,\xi_n],\Xi(m,n);[z_a,z_b],Z(a,b))\Big{)} .
\end{align*}
As a consequence, these expressions will coincide if, and only if 
\begin{eqnarray*}
\rho^{(r+1)}(\hat{t}_p(\xi_j))\rho_0^1(\mu(z_k))\omega(\Xi(j);Z(k))=\rho_0^1(\mu(z_k))\rho^{(r)}(\hat{t}_p(\xi_j))\omega(\Xi(j);Z(k)).
\end{eqnarray*}
Indeed, for all $y\in\hh$, $z\in\gg$ and $X=(x_1,...,x_r)\in\gg^r$, using the fact that $\rho_0^1$ is a Lie algebra homomorphism, we have got
\begin{align*}
 \rho^{(r+1)}(y)\rho_0^1(\mu(z))\omega(\Xi(j);X) & = \rho_0^1(y)\rho_0^1(\mu(z))\omega(\Xi(j);X)-\rho_0^1(\mu(\Lie_yz))\omega(\Xi(j);X) + \\
                                               & \qquad\qquad -\sum_{k=1}^r\rho_0^1(\mu(z))\omega(\Xi(j);x_1,...,\Lie_yx_k,...) \\
                                               & =\rho_0^1(y)\rho_0^1(\mu(z))\omega(\Xi(j);X)-\rho_0^1([y,\mu(z)])\omega(\Xi(j);X) + \\
                                               & \qquad\qquad -\sum_{k=1}^r\rho_0^1(\mu(z))\omega(\Xi(j);x_1,...,\Lie_yx_k,...) \\
                                               & =\rho_0^1(\mu(z))\rho_0^1(y)\omega(\Xi(j);X)-\rho_0^1(\mu(z))\sum_{k=1}^r\omega(\Xi(j);x_1,...,\Lie_yx_k,...) \\
                                               & =\rho_0^1(\mu(z))\rho^{(r)}(y)\omega(\Xi(j);X),
\end{align*}
so the result follows.

\end{proof}
\begin{prop}\label{q-pagDblCx}
$(\partial\delta_{(1)}-\delta_{(1)}\partial)\omega=0\in C^{p+1,q}_{r+1}(\gg_1,\phi)$; thus, for constant $q$, there is an actual double complex that we call the $q$-page.
\end{prop}
\begin{proof}
Let $\Xi=(\xi_1,...,\xi^q)\in\gg_{p+1}^q$ and $Z=(z_0,...,z^r)\in\gg^{r+1}$, then
\begin{align*}
\partial\delta_{(1)}\omega(\Xi;Z) & =\sum_{k=0}^{p+1}(-1)^k\delta_{(1)}\omega(\partial_k\Xi;Z) \\
                                  & =\sum_{k=0}^{p+1}(-1)^k\Big{(}\sum_{j=0}^r(-1)^j\rho_0^1(\mu(z_j))\omega(\partial_k\Xi;Z(j))+\sum_{m<n}(-1)^{m+n}\omega(\partial_k\Xi;[z_m,z_n],Z(m,n))\Big{)} \\
                                  & =\sum_{j=0}^r(-1)^j\rho_0^1(\mu(z_j))\sum_{k=0}^{p+1}(-1)^k\omega(\partial_k\Xi;Z(j))+ \\
                                  & \qquad\quad +\sum_{m<n}(-1)^{m+n}\sum_{k=0}^{p+1}(-1)^k\omega(\partial_k\Xi;[z_m,z_n],Z(m,n)) \\
                                  & =\sum_{j=0}^r(-1)^j\rho_0^1(\mu(z_j))\partial\omega(\Xi;Z(j))+\sum_{m<n}(-1)^{m+n}\partial\omega(\Xi;[z_m,z_n],Z(m,n)) =\delta_{(1)}\partial\omega(\Xi;Z)
\end{align*}
\end{proof}
As it was pointed out, not all the differentials in the complexes commute with one another; indeed, $\delta^{(r)}$ needs not commute with $\partial$. For constant $r$, we call the resulting two dimensional grid an $r$-page. Essentially, $r$-pages coincide with the double complex of $\gg_1$ tensored with $\bigwedge^r\gg^*\otimes W$. For $r=0$, $\delta^{(r)}$ and $\partial$ commute up homotopy. More specifically, when evaluated, the resulting elements in $V$ are isomorphic in the $2$-vector space.  
\begin{prop}\label{AlgUpToHomotopy r=0}
Let $\omega\in\bigwedge^q\gg_p^*\otimes V$ and $\Xi=(\xi_0,...,\xi_q)\in\gg_{p+1}^{q+1}$. Then
\begin{eqnarray*}
\delta\partial\omega(\Xi)=\partial\delta\omega(\Xi)+\phi\Big{(}\sum_{j=0}^q(-1)^j\rho_1(x_j^0)\omega(\partial_0\Xi(j))\Big{)}=\partial\delta\omega(\Xi)+\Delta\delta_{(1)}\omega(\Xi)
\end{eqnarray*}
\end{prop}
\begin{proof}
Identifying $\xi_j$ with $(x_j^0,...,x_j^p;y_j)$, we have got further identifications
\begin{eqnarray*}
\partial _k (\xi_j)\sim
  \begin{cases}
    (x_j^1,...,x_j^p;y_j)                     & \quad \text{if } k=0  \\
    (x_j^0,...,x_j^{k-1}+x_j^k,...,x_j^p;y_j) & \quad \text{if } 0<k\leq p\\
    (x_j^0,...,x_j^{p-1};y_j+\mu(x_j^p))      & \quad \text{if } k=p+1. 
  \end{cases}
\end{eqnarray*}
Now, computing
\begin{align*}
\delta\partial\omega(\Xi) & =\sum_{j=0}^{q}(-1)^j\rho_0^0(\hat{t}_{p+1}(\xi_j))\partial\omega(\Xi(j))+\sum_{m<n}(-1)^{m+n}\partial\omega([\xi_m,\xi_n],\Xi(m,n)) \\
						  & =\sum_{j=0}^{q}(-1)^j\rho_0^0(\hat{t}_{p+1}(\xi_j))\sum_{k=0}^{p+1}(-1)^k\omega(\partial_k\Xi(j))+\sum_{m<n}(-1)^{m+n}\sum_{k=0}^{p+1}(-1)^k\omega(\partial_k[\xi_m,\xi_n],\partial_k(\Xi(m,n)));
\end{align*}
while on the other hand,
\begin{align*}
\partial\delta\omega(\Xi) & =\sum_{k=0}^{p+1}(-1)^k\delta\omega(\partial_k\Xi) \\
						  & =\sum_{k=0}^{p+1}(-1)^k\Big{(}\sum_{j=0}^{q}(-1)^j\rho_0^0(\hat{t}_p(\partial_k\xi_j))\omega(\partial_k\Xi(j))+\sum_{m<n}(-1)^{m+n}\omega([\partial_k\xi_m,\partial_k\xi_n],(\partial_k\Xi)(m,n))\Big{)}.
\end{align*}
Notice that, again $\partial_k(\Xi(m,n))=(\partial_k\Xi)(m,n)$ trivially; therefore, since $\partial_k$ is a Lie algebra homomorphism for each $k$, the second term in the expressions above coincide. As for the first term, from the identifications above, one sees that 
\begin{eqnarray*}
\hat{t}_p(\partial_k\xi_j)=
  \begin{cases}
    y_j+\sum_{r=1}^p\mu(x_j^r)=\hat{t}_{p+1}(\xi_j)-\mu(x_j^0) & \quad \text{if } k=0  \\
    y_j+\sum_{r=0}^p\mu(x_j^r)=\hat{t}_{p+1}(\xi_j)            & \quad \text{otherwise}. 
  \end{cases}
\end{eqnarray*}
This together with the relation $\rho_0^0(\mu(x))=\phi\rho_1(x)$ coming from the $2$-representation imply that the difference is 
\begin{align*}
    (\delta\partial-\partial\delta)\omega(\Xi) & =\sum_{j=0}^q(-1)^j\rho_0^0(\mu(x_j^0))\omega(\partial_0\Xi(j)) \\
                                               & =\phi\Big{(}\sum_{j=0}^q(-1)^j\rho_1(x_j^0)\omega(\partial_0\Xi(j))\Big{)},
\end{align*}
as desired. 

\end{proof}
The previous proposition shows that it is possible for the $r$-pages to be double complexes. For instance, if $W=\lbrace 0\rbrace$, a $2$-representation amounts to a usual representation $\rho$ of $\hh/\mu(\gg)$ on $V$ (cf. example \ref{unitRep}) and the vanishing of the representations on the ideal $\mu(\gg)$ implies the page $r=0$ is a double complex. Another instance in which this holds is if the $2$-representation takes values on a Lie group bundle internal to the category of vector spaces, i.e. on a $2$-vector space $\xymatrix{W \ar[r]^0 & V}$. \\
For $r\neq 0$, the difference $\partial\delta^{(r)}-\delta^{(r)}\partial$ is not necessarily zero either. Instead, the next proposition shows that the compositions are homotopic as maps of complexes. 
\begin{prop}\label{starTop}
For $r\neq 0$,
\begin{eqnarray*}
\delta^{(r)}\partial-\partial\delta^{(r)}=\delta_{(1)}\circ\Delta+\Delta\circ\delta_{(1)}.
\end{eqnarray*}
\end{prop}
\begin{proof}
Just as in the proof of proposition \ref{AlgUpToHomotopy r=0} and the argument right after, we compute
\begin{align*}
(\delta^{(r)}\partial-\partial\delta^{(r)})\omega(\Xi ;Z) & =\sum_{j=0}^q(-1)^j\rho^{(r)}(\mu(x_j^0))\omega(\partial_0\Xi(j);Z) \\
                                          & =\sum_{j=0}^q(-1)^j\Big{(}\rho_0^1(\mu(x_j^0))\omega(\partial_0\Xi(j);Z)-\sum_{k=1}^r\omega(\partial_0\Xi(j);z_1,...,[x_j^0,z_k],...,z_r)\Big{)}.
\end{align*}
On the other hand, for $r=1$, 
\begin{align*}
    \delta_{(1)}\Delta\omega(\Xi;x) & =\rho_1(x)\Delta\omega(\Xi) \\
                                    & =\rho_1(x)\phi\Big{(}\sum_{j=0}^q(-1)^j\omega(\partial_0\Xi(j);x_j^0)\Big{)};
\end{align*}
hence, using the fact that $\rho_1(x)\phi=\rho_0^0(\mu(x))$,
\begin{align*}
    (\delta'\partial-\partial\delta' & -\delta_{(1)}\Delta)\omega(\Xi;x) \\
    & =\sum_{j=0}^q(-1)^j\Big{(}\rho_0^1(\mu(x_j^0))\omega(\partial_0\Xi(j);x)-\rho_0^0(\mu(x))\omega(\partial_0\Xi(j);x_j^0)-\omega(\partial_0\Xi(j);[x_j^0,x])\Big{)} \\
    & =\sum_{j=0}^q(-1)^j\delta_{(1)}\omega(\partial_0\Xi(j);x_j^0,x)=\Delta\delta_{(1)}\omega(\Xi;x).
\end{align*}
Whereas, for $r>1$, 
\begin{align*}
\delta_{(1)}\Delta\omega(\Xi;Z) & =\sum_{k=1}^r(-1)^k\rho_0^1(\mu(z_k))\Delta\omega(\Xi;Z(k))+\sum_{m<n}(-1)^{m+n}\Delta\omega(\Xi;[z_m,z_n],Z(m,n)) \\
                                & =\sum_{k=1}^r(-1)^k\rho_0^1(\mu(z_k))\sum_{j=0}^q(-1)^j\omega(\partial_0\Xi(j);x_j^0,Z(k))+ \\
                                & \qquad\qquad +\sum_{m<n}(-1)^{m+n}\sum_{j=0}^q(-1)^j\omega(\partial_0\Xi(j);x_j^0,[z_m,z_n],Z(m,n));
\end{align*}
hence, by introducing the vectors
\begin{eqnarray*}
Z^j=(x_j^0,z_1,...,z_r)\in\gg^{r+1},
\end{eqnarray*}
and rearranging, we have got
\begin{align*}
(\delta^{(r)} & \partial-\partial\delta^{(r)}-\delta_{(1)}\Delta)\omega(\Xi;Z) \\
    & =\sum_{j=0}^q(-1)^j\Big{(}\rho_0^1(\mu(x_j^0))\omega(\partial_0\Xi(j);Z)+\sum_{k=1}^r(-1)^{k+1}\rho_0^1(\mu(z_k))\omega(\partial_0\Xi(j);x_j^0,Z(k))+ \\
    &\qquad\qquad -\sum_{k=1}^r(-1)^{k-1}\omega(\partial_0\Xi(j);[x_j^0,z_k],Z(k))+\sum_{m<n}(-1)^{m+n+2}\omega(\partial_0\Xi(j);[z_m,z_n],x_j^0,Z(m,n))\Big{)} \\
    & =\sum_{j=0}^q(-1)^j\Big{(}\sum_{k=0}^r(-1)^{k}\rho_0^1(\mu(z^j_k))\omega(\partial_0\Xi(j);Z^j(k))+\sum_{m<n}(-1)^{m+n}\omega(\partial_0\Xi(j);[z^j_m,z^j_n],Z^j(m,n))\Big{)} \\
    & =\sum_{j=0}^q(-1)^j\delta_{(1)}\omega(\partial_0\Xi(j);Z^j)=\Delta\delta_{(1)}\omega(\Xi;Z),
\end{align*}
as desired. 

\end{proof}
We schematize the relation of the previous proposition by the following diagram:
\begin{eqnarray*}
\xymatrix{
\circ \ar@{.}[rr]\ar@{.}[dr] & & C^{p,q+1}_r(\gg_1,\phi) \ar@{.}[rr]\ar[dr]^\partial & & \circ\ar@{.}[dr] & \\
 & C^{p+1,q+1}_{r-1}(\gg_1,\phi) \ar[rr]^{\delta_{(1)}\qquad\quad} & & C^{p+1,q+1}_r(\gg_1,\phi)\ar@{.}[rr] & & \circ \\
\circ \ar@{.}[rr]\ar@{.}[dr]\ar@{.}[uu] & & C^{p,q}_r(\gg_1,\phi)\ar'[r]^{\quad\delta_{(1)}}[rr]\ar[dr]^\partial\ar[ul]^\Delta\ar'[u]_{\delta^{(r)}}[uu] & & C^{p,q}_{r+1}(\gg_1,\phi)\ar@{.}[dr]\ar@{.}[uu]\ar[ul]_\Delta & \\
 & \circ \ar@{.}[rr]\ar@{.}[uu] & & C^{p+1,q}_r(\gg_1,\phi)\ar@{.}[rr]\ar[uu]_(.35){\delta^{(r)}} & & \circ \ar@{.}[uu]
}
\end{eqnarray*}







Let $\omega\in C^{p,q}_r(\gg_1,\phi)$, without caring much about signs for the time being and disregarding $\Delta_k$ for $k>1$, we have got the following diagram roughly representing the relations that are to vanish if $\nabla$ is indeed the differential of a complex:
\begin{eqnarray*}
\xymatrix{
            &                                         & \Delta^2                                    &                                             &  \\
            & [\partial,\Delta]                       & \Delta\omega \ar[l]\ar[u]\ar[r]\ar[d]       & [\delta^{(r)},\Delta]                       &  \\
\partial^2  & \partial\omega \ar[l]\ar[u]\ar[r]\ar[d] & [\partial,\delta^{(r)},\delta_{(1)},\Delta] & \delta^{(r)}\omega \ar[l]\ar[u]\ar[r]\ar[d] & (\delta^{(r)})^2  \\
            & [\partial,\delta_{(1)}]                 & \delta_{(1)}\omega \ar[l]\ar[u]\ar[r]\ar[d] & [\delta^{(r)},\delta_{(1)}]                 &  \\
            &                                         & \delta_{(1)}^2                              &                                             &  
}
\end{eqnarray*}
We have shown that six of these relations already; namely,
\begin{itemize}
    \item $\partial^2\omega=0\in C^{p+2,q}_r(\gg_1,\phi)$, because $\partial$ is the differential of groupoid cochains,
    \item $(\delta^{(r)})^2\omega=0\in C^{p,q+2}_r(\gg_1,\phi)$, because $\delta^{(r)}$ is a Chevalley-Eilenberg differential,
    \item $\delta_{(1)}^2\omega=0\in C^{p,q}_{r+2}(\gg_1,\phi)$, again, because $\delta_{(1)}$ is a Chevalley-Eilenberg differential, \item $(\partial\delta_{(1)}-\delta_{(1)}\partial)\omega=0\in C^{p+1,q}_{r+1}(\gg_1,\phi)$ is the contents of proposition \ref{q-pagDblCx},
    \item $(\delta^{(r)}\delta_{(1)}-\delta^{(r+1)}\delta_{(1)})\omega=0\in C^{p,q+1}_{r+1}(\gg_1,\phi)$ is the contents of proposition \ref{p-pagDblCx} and
    \item $(\delta^{(r)}\partial-\partial\delta^{(r)}-\delta_{(1)}\circ\Delta-\Delta\circ\delta_{(1)})\omega=0\in C^{p+1.q+1}_r(\gg_1,\phi)$ is the contents of proposition \ref{starTop}.
\end{itemize}
We can assert one more:
\begin{prop}\label{partialDelta}
$(\partial\Delta+\Delta\partial)\omega=0\in C^{p+2,q+1}_{r-1}(\gg_1,\phi)$.
\end{prop}
\begin{proof}
Let $Z\in\gg^{r+1}$, $\Xi=(\xi_0,...,\xi^q)\in\gg_{p+2}^{q+1}$ and let $\xi_j$ be identified with $(x_j^0,...,x_j^{p+1};y_j)$. Computing,
\begin{align*}
(\partial\Delta+\Delta\partial)\omega(\Xi;Z) & =\sum_{j=0}^q(-1)^j\partial\omega(\partial_0\Xi(j);x_j^0,Z)+\sum_{k=0}^{p+2}(-1)^k\Delta\omega(\partial_k\Xi;Z) \\
             & =\sum_{j=0}^q(-1)^j\sum_{k=0}^{p+1}(-1)^k\omega(\partial_k(\partial_0\Xi(j));x_j^0,Z)+\sum_{j=0}^q(-1)^j\omega(\partial_0(\partial_0\Xi)(j);x_j^1,Z) \\
             & \quad -\sum_{j=0}^q(-1)^j\omega(\partial_0(\partial_1\Xi)(j);x_j^0+x_j^1,Z)+\sum_{k=2}^{p+2}(-1)^k\sum_{j=0}^q(-1)^j\omega(\partial_0(\partial_k\Xi)(j);x_j^0,Z) \\
             & =\sum_{j=0}^q(-1)^j\Big{(}\sum_{k=0}^{p+1}(-1)^k\omega(\partial_k\partial_0\Xi(j);x_j^0,Z)-\omega(\partial_0(\partial_0\Xi)(j);x_j^0,Z) \\
             & \qquad\qquad +\sum_{k=2}^{p+2}(-1)^k\omega(\partial_0(\partial_k\Xi)(j);x_j^0,Z)\Big{)}.
\end{align*}
Now, either computing, or from the simplicial identities, we know that
\begin{eqnarray*}
\partial_k\partial_0(\xi_j)\sim
  \begin{cases}
    (x_j^2,...,x_j^{p+1};y_j)                       & \quad \text{if } k=0  \\
    (x_j^1,...,x_j^{k}+x_j^{k+1},...,x_j^{p+1};y_j) & \quad \text{if } 0<k\leq p\\
    (x_j^1,...,x_j^{p};y_j+\mu(x_j^{p+1}))          & \quad \text{if } k=p+1, 
  \end{cases}
\end{eqnarray*}
and
\begin{eqnarray*}
\partial_0\partial_k(\xi_j)\sim
  \begin{cases}
    (x_j^2,...,x_j^{p+1};y_j)                     & \quad \text{if } k\in\lbrace 0,1\rbrace  \\
    (x_j^1,...,x_j^k+x_j^{k+1},...,x_j^{p+1};y_j) & \quad \text{if } 1<k\leq p+1\\
    (x_j^1,...,x_j^{p};y_j+\mu(x_j^{p+1}))        & \quad \text{if } k=p+2; 
  \end{cases}
\end{eqnarray*}
therefore, the sums cancel one another and the result follows.

\end{proof}
After inspecting the degrees, one sees that the remaining two relations need to take $\Delta_2$ into account. Successively adding one difference map at a time one arrives at the following rephrasing of theorem \ref{The2AlgCx}.
\begin{theorem}\label{algDiffs}
The higher difference maps $\Delta_k$ verify the following set of equations for elements in $C^{p,q}_r(\gg_1,\phi)$
\begin{itemize}
    \item[i)] $\delta^{(r-k)}\Delta_k+\Delta_k\delta^{(r)}=\delta_{(1)}\Delta_{k+1}+\Delta_{k+1}\delta_{(1)}$ for all $1\leq k<r$,
    \item[ii)]  $\delta\Delta_r+\Delta_r\delta^{(r)}=\Delta_{r+1}\delta_{(1)}$ and
    \item[iii)] assuming the convention that $\Delta_0=\partial$,
    \begin{eqnarray*}
    \sum_{i=0}^k\Delta_{k-i}\Delta_i=0
    \end{eqnarray*}
    for all $1\leq k\leq r$.
\end{itemize}
\end{theorem}
We lay down the computations that show some of these equations in the sequel. 

\subsubsection*{Higher difference maps}
We start by studying how the relation labeled $[\delta^{(r)},\Delta]$ behaves in the special case $r=1$. \\
Let $\alpha\in\bigwedge^q\gg_p^*\otimes\gg^*\otimes W$, $\Xi=(\xi_0,...,\xi_{q+1})\in\gg_{p+1}^{q+2}$ with the by now usual identification $\xi_j\sim(x_j^0,...,x_j^p;y_j)$, and let us compute separately $\delta\Delta\alpha$ and $\Delta\delta'\alpha$. On the one hand, we have got
\begin{align*}
\delta\Delta\alpha(\Xi) & =\sum_{k=0}^{q+1}(-1)^k\rho_0^0(\hat{t}_{p+1}(\xi_k))\Delta\alpha(\Xi(k))+\sum_{m<n}(-1)^{m+n}\Delta\alpha([\xi_m,\xi_n],\Xi(m,n)) .
\end{align*}
The first term in this expression is computed to be
\begin{align*}
 & \sum_{k=0}^{q+1}(-1)^k\rho_0^0(\hat{t}_{p+1}(\xi_k))\phi\Big{[}\sum_{j=0}^{k-1}(-1)^j\alpha(\partial_0\Xi(j,k);x_j^0)+\sum_{j=k+1}^{q+1}(-1)^{j+1}\alpha(\partial_0\Xi(k,j);x_j^0)\Big{]}.
\end{align*}
The second term is trickier though. Carefully computing it, one realizes it coincides with
\begin{align*}
 & \sum_{m<n}(-1)^{m+n}\phi\Big{(}\alpha(\partial_0\Xi(m,n);[x_m^0,x_n^0]+\Lie_{\hat{t}_p(\partial_0\xi_m)}x_n^0-\Lie_{\hat{t}_p(\partial_0\xi_n)}x_m^0)+ \\
 & \qquad +\sum_{j=0}^{m-1}(-1)^{j+1}\alpha(\partial_0([\xi_m,\xi_n]),\partial_0\Xi(j,m,n);x_j^0)+\sum_{j=m+1}^{n-1}(-1)^j\alpha(\partial_0([\xi_m,\xi_n]),\partial_0\Xi(m,j,n);x_j^0)+ \\
 & \qquad\qquad +\sum_{j=n+1}^q(-1)^{j+1}\alpha(\partial_0([\xi_m,\xi_n]),\partial_0\Xi(m,n,j);x_j^0)\Big{)} 
\end{align*}
The shape of the first term comes from the definition of the Lie algebra structure on $\gg_{p+1}$. Recall that the bracket is inherited from the product $\gg^{p+1}$; therefore, the first entry of $[\xi_m,\xi_n]$ is
\begin{align*}
\Big{[}\big{(}x_m^0,y_m & +\sum_{k=1}^p\mu(x_m^k)\big{)},\big{(}x_n^0,y_n+\sum_{k=1}^p\mu(x_n^k)\big{)}\Big{]}_1 \\
 & =\Big{(}[x_m^0,x_n^0]+\Lie_{y_m+\sum_{k=1}^p\mu(x_m^k)}x_n^0-\Lie_{y_n+\sum_{k=1}^p\mu(x_n^k)}x_m^0,\big{[}y_m+\sum_{k=1}^p\mu(x_m^k),y_n+\sum_{k=1}^p\mu(x_n^k)\big{]}\Big{)},
\end{align*}
and we already saw that $\hat{t}_p(\partial_0\xi_j)=y_j+\sum_{k=1}^p\mu(x_j^k)$. The second term is written in such a way that it has got the appropriate signs when written with our conventions. On the other hand, 
\begin{align*}
\Delta\delta'\alpha(\Xi) & =\phi\Big{(}\sum_{j=0}^q(-1)^j\delta'\alpha(\partial_0\Xi(j);x_j^0)\Big{)}.
\end{align*}
 After evaluating the differential $\delta'$, there will be two terms. The first will be
\begin{align*}         
\sum_{j=0}^{q+1}(-1)^j &  \phi\Big{(}\sum_{k=0}^{j-1}(-1)^k\rho'(\hat{t}_p(\partial_0\xi_k))\alpha(\partial_0\Xi(k,j);x_j^0)+\sum_{k=j+1}^q(-1)^{k+1}\rho'(\hat{t}_p(\partial_0\xi_k))\alpha(\partial_0\Xi(j,k);x_j^0)\Big{)}
\end{align*}
while the second,
\begin{align*}
\sum_{j=0}^{q+1}(-1)^j & \phi\Big{(}\sum_{m=0}^{j-1}\Big{[}\sum_{n=m+1}^{j-1}(-1)^{n}\alpha([\partial_0\xi_m,\partial_0\xi_n]),\partial_0\Xi(m,n,j);x_j^0)+ \\
 & \qquad +\sum_{n=j+1}^{q+1}(-1)^{n+1}\alpha([\partial_0\xi_m,\partial_0(\xi_n]),\partial_0\Xi(m,j,n);x_j^0)\Big{]}+ \\
 & \qquad\qquad +\sum_{m=j+1}^{q+1}\sum_{n=m+1}^{q+1}(-1)^{n}\alpha([\partial_0\xi_m,\partial_0\xi_n],\partial_0\Xi(m,n,j);x_j^0)\Big{)}. 
\end{align*}
When considering the sum of the whole of the two expressions, this second term will cancel out with the second part of the second term above simply due to the fact that $\partial_0$ is a Lie algebra homomorphism. Furthermore, since $\rho_0^0(y)\phi=\phi\rho_0^1(y)$ for all $y\in\hh$, we will have again the difference 
\begin{eqnarray*}
\rho_0^1(\hat{t}_{p}(\xi_k))-\rho_0^1(\hat{t}_p(\partial_0\xi_k))=\rho_0^1(\mu(x_k^0)),
\end{eqnarray*} 
and in the meantime, the second terms of the representation $\rho'$,
\begin{eqnarray*}
\alpha(\partial_0\Xi(m,n);\Lie_{\hat{t}_p(\partial_0\xi_m)}x_n^0)
\end{eqnarray*}
will cancel the ones coming from the Lie bracket in $\gg_{p+1}$. Thus, we are left behind with a series of terms of the form
\begin{eqnarray*}
\rho_0^1(\mu(x_k))\alpha(\partial_0\Xi(j,k);x_j^0)
\end{eqnarray*}
and another series of terms of the form
\begin{eqnarray*}
\alpha(\partial_0\Xi(m,n);[x_m^0,x_n^0]).
\end{eqnarray*}
What is remarkable is that these come with the right signs to build up differentials in the image of $\delta_{(1)}$. In fact, the preceding discussion can be summarized in the following lemma.
\begin{lemma}\label{firstHigher}
Let 
\begin{eqnarray*}
\xymatrix{
\Delta_2:\bigwedge^q\gg_p^*\otimes\bigwedge^2\gg^*\otimes W \ar[r] & \bigwedge^{q+2}\gg_{p+1}^*\otimes V
}
\end{eqnarray*}
be defined by
\begin{eqnarray*}
\Delta_2\alpha(\Xi):=\phi\Big{(}\sum_{m<n}(-1)^{m+n}\alpha(\partial_0\Xi(m,n);x_m^0,x_n^0)\Big{)}.
\end{eqnarray*}
Then
\begin{eqnarray*}
\delta\Delta+\Delta\delta'=\Delta_2\circ\delta_{(1)}.
\end{eqnarray*}
\end{lemma}
The second difference maps will help us further decide the behaviour of the missing relations, as shown by the following proposition.
\begin{prop}
For $r>2$, 
\begin{eqnarray*}
\delta^{(r-1)}\circ\Delta+\Delta\circ\delta^{(r)}=\delta_{(1)}\circ\Delta_2-\Delta_2\circ\delta_{(1)}
\end{eqnarray*}
and
\begin{eqnarray*}
\Delta^2=-(\Delta_2\circ\partial+\partial\circ\Delta_2).
\end{eqnarray*}
\end{prop}
\begin{proof}
Let $Z=(z_1,...,z_{r-1})\in\gg^{r-1}$, $\Xi=(\xi_0,...,\xi_{q+1})\in\gg_{p+1}^{q+2}$ and identify $\xi_j$ with $(x_j^0,...,x_j^p;y_j)$. We write the expressions for the terms of the left hand side of the first equation
\begin{align*}
\delta^{(r-1)}\Delta\omega(\Xi;Z) & =\sum_{k=0}^{q+1}(-1)^k\rho^{(r-1)}(\hat{t}_{p+1}(\xi_k))\Delta(\Xi(k);Z)+\sum_{m<n}(-1)^{m+n}\Delta([\xi_m,\xi_n],\Xi(k);Z), \\
\Delta\delta^{(r)}\omega(\Xi;Z)   & =\sum_{j=0}^{q+1}(-1)^j\delta^{(r)}\omega(\partial_0\Xi(j);x_j^0,Z).
\end{align*}
For each pair $(j,k)$, if $j>k$, there are terms 
\begin{eqnarray*}
(-1)^k\rho^{(r-1)}(\hat{t}_{p+1}(\xi_k))\big{[}(-1)^{j+1}\omega(\partial_0\Xi(k,j);x_j^0,Z)\big{]} & \textnormal{and} &
(-1)^{j}(-1)^k\rho^{(r)}(\hat{t}_p(\partial_0\xi_k))\omega(\partial_0\Xi(k,j);x_j^0,Z)
\end{eqnarray*}
in each of the equations above respectively. If, on the other hand, $j<k$, one has got the following terms: 
\begin{eqnarray*}
(-1)^k\rho^{(r-1)}(\hat{t}_{p+1}(\xi_k))\big{[}(-1)^{j}\omega(\partial_0\Xi(j,k);x_j^0,Z)\big{]} & \textnormal{and} &
(-1)^{j}(-1)^{k+1}\rho^{(r)}(\hat{t}_p(\partial_0\xi_k))\omega(\partial_0\Xi(j,k);x_j^0,Z).
\end{eqnarray*}
We write the sums of these pairs of terms, 
\begin{align*}
\sum_{\tau\in S_2} & (-1)^{(j+k)}  \abs{\tau}\big{[}\rho^{(r-1)}(\hat{t}_{p+1}(\xi_k))-\rho^{(r)}(\hat{t}_p(\partial_0\xi_k))\big{]}\omega(\partial_0\Xi(\tau(k,j));x_j^0,Z) \\
             & =\sum_{\tau\in S_2}(-1)^{(j+k+1)}\abs{\tau}\big{[}\omega(\partial_0\Xi(\tau(k,j));\Lie_{\hat{t}_p(\partial_0\xi_k)}x_j^0,Z)+ \\
             & \qquad +\rho_0^1(\mu(x_k^0))\omega(\partial_0\Xi(\tau(k,j));x_j^0,Z)-\sum_{i=1}^{r-1}\omega(\partial_0\Xi(\tau(k,j));x_j^0,z_1,...,[x_k^0,x_i],...,x_{r-1})\big{]}.
\end{align*}
The second sum $\delta^{(r-1)}\Delta\omega(\Xi;Z)$, when expanding $\Delta$ will have first term 
\begin{eqnarray*}
\sum_{m<n}(-1)^{m+n}\omega(\Xi(m,n);[x_m^0,x_n^0]+\Lie_{\hat{t}_p(\partial_0\xi_m)}x_n^0-\Lie_{\hat{t}_p(\partial_0\xi_n)}x_m^0,Z).
\end{eqnarray*}
All the remaining terms will cancel out in a similar fashion to the discussion preceding the statement of lemma \ref{firstHigher}. In so, 
\begin{align*}
(\delta^{(r-1)}\Delta & +\Delta\delta^{(r)})\omega(\Xi; Z) \\
  & =-\sum_{m<n}(-1)^{m+n}\Big{[}\rho_0^1(\mu(x_m^0))\omega(\partial_0\Xi(m,n);x_n^0,Z)-\rho_0^1(\mu(x_m^0))\omega(\partial_0\Xi(m,n);x_n^0,Z)+ \\
  & \qquad -\omega(\partial_0\Xi(m,n);[x_m^0,x_n^0],Z)+ \\
  & \qquad\quad+\sum_{k=1}^{r-1}(-1)^{k+1}\big{(}\omega(\partial_0\Xi(m,n);[x_m^0,z_k],x_n^0,Z(k))-\omega(\partial_0\Xi(m,n);[x_n^0,z_k],x_n^0,Z(k))\big{)}\Big{]}
\end{align*}
Now, 
\begin{align*}
\delta_{(1)}\Delta_2\omega(\Xi;Z) & =\sum_{k=1}^{r-1}(-1)^{k+1}\rho_0^1(\mu(z_k))\Delta_2\omega(\Xi;Z(k))+\sum_{a<b}(-1)^{a+b}\Delta_2\omega(\Xi;[z_a,z_b],Z(a,b)) \\
                                  & =\sum_{k=1}^{r-1}(-1)^{k+1}\sum_{m<n}(-1)^{m+n}\rho_0^1(\mu(z_k))\omega(\partial_0\Xi(m,n);x_m^0,x_n^0,Z(k))+ \\
                                  & \qquad\quad +\sum_{a<b}(-1)^{a+b}\sum_{m<n}(-1)^{m+n}\omega(\partial_0\Xi(m,n);x_m^0,x_n^0,[z_a,z_b],Z(a,b)) ;
\end{align*}
thus, one sees that 
\begin{eqnarray*}
(\delta^{(r-1)}\Delta+\Delta\delta^{(r)}-\delta_{(1)}\Delta_2)\omega(\Xi; Z)=-\Delta_2\delta_{(1)}\omega(\Xi; Z),
\end{eqnarray*}
and the first equation of the statement holds. \\
As for the second equation, let $\Xi=(\xi_0,...,\xi_{q+1})\in\gg_{p+2}^{q+2}$ with $\xi_j\sim(x_j^0,...,x_j^{p+1};y_j)$ and $Z$ as before. Computing,
\begin{align*}
\Delta_2\partial\omega(\Xi;Z) & =\sum_{m<n}(-1)^{m+n}\partial\omega(\partial_0\Xi(m,n);x_m^0,x_n^0,Z) \\
                              & =\sum_{m<n}(-1)^{m+n}\sum_{k=0}^{p+1}(-1)^k\omega(\partial_k\partial_0\Xi(m,n);x_m^0,x_n^0,Z);
\end{align*}
whereas, on the other hand, 
\begin{align*}
\partial\Delta_2\omega(\Xi;Z) & =\sum_{k=0}^{p+2}(-1)^k\Delta_2\omega(\partial_k\Xi;Z) \\
                              & = \sum_{m<n}(-1)^{m+n}\Big{(}\omega(\partial_0\partial_0\Xi(m,n);x_m^1,x_n^1,Z)-\omega(\partial_0\partial_1\Xi(m,n);x_m^0+x_m^1,x_n^0+x_n^1,Z)\Big{)}+ \\
                              & \qquad +\sum_{k=2}^{p+2}(-1)^k\sum_{m<n}(-1)^{m+n}\omega(\partial_0\partial_k\Xi(m,n);x_m^0,x_n^0,Z).
\end{align*}
Due to the simplicial identities that we outlined in the proof of proposition \ref{partialDelta}, we know that $\partial_0\partial_0=\partial_0\partial_1$. By setting
\begin{eqnarray*}
\varpi_{mn}(w_1,w_2):=\omega(\partial_0\partial_0\Xi(m,n);w_1,w_2,Z),
\end{eqnarray*}
we can write the sum
\begin{align*}
(\Delta_2\partial+\partial\Delta_2)\omega(\Xi;Z) & =\sum_{m<n}(-1)^{m+n}\Big{(}\varpi(x_m^1,x_n^1)+\varpi(x_m^1,x_n^1)-\varpi(x_m^0+x_m^1,x_n^0+x_n^1)\Big{)} \\
                                                 & =\sum_{m<n}(-1)^{m+n}\Big{(}\varpi(x_n^1,x_m^0)-\varpi(x_m^1,x_n^0)\Big{)}
\end{align*}
which is precisely $-\Delta^2(\Xi;Z)$. 

\end{proof}
We would like to close this section by pointing out that theorem \ref{The2AlgCx} implies that there is a double complex that allows us to regard the Lie $2$-algebra cohomology with values in a $2$-representation as the total cohomology of a double complex, just as in the case of trivial coefficients. Indeed, we can think of the difference maps as being maps of double complexes between the $p$-pages (cf.proposition \ref{p-pagDblCx}). By considering the total complex of the $p$-pages, we get an honest double complex, whose columns are the total complexes of the $p$-pages and whose horizontal differentials are the maps of complexes induced by the difference maps. The total complex of this double complex (and its cohomology) coincides with that of theorem \ref{The2AlgCx}.


\section[2-cohomology of Lie 2-algebras]{2-cohomology of Lie 2-algebras}\label{2AlgCoh}
\sectionmark{Cohomology}

\subsection{2-cohomology with trivial coefficients}\label{AlgDcx}

We give an interpretation of $H^2_{tot}(\gg_1 )$. A $2$-cocycle consists of a pair of functions $(\omega,\varphi)\in (\hh^*\wedge\hh^*)\oplus \gg_1^*$ such that:\\
1) $\delta \omega = 0$, i.e. $-\omega([y_0,y_1] ,y_2)+\omega([y_0 ,y_2], y_1)-\omega([y_1, y_2],y_0)=0$ for all triples $y_0,y_1,y_2\in \hh$ \\
2) $\partial \varphi =0$, i.e. $\varphi(x_2 ,y)-\varphi(x _1 +x _2 ,y)+\varphi(x_1 ,y+\mu (x_2))=0$ for all $x_1,x_2\in\gg$ and $y\in\hh$. Due to linearity, this boils down to $\varphi(0,y+\mu(x_2))=0$ for all $y\in\hh$ and $x_2\in\gg$.\\
3) $\partial \omega+\delta \varphi=0$, i.e. $\omega(y_0,y_1)-\omega(y_0+\mu (x_0),y_1+\mu (x_1))= \varphi([(x_0,y_0),(x_1,y_1)]_\Lie)$ for all pairs $(x_0,y_0),(x_1,y_1)\in\gg\oplus\hh$. Again, due to bilinearity of $\omega$ and linearity of $\varphi$, this equation can be rewritten as
\begin{eqnarray*}
-\omega(y_0,\mu(x_1))-\omega(\mu(x_0),\mu(x_1))-\omega(\mu(x_0),y_1)=\varphi(\Lie_{y_0}x_1-\Lie_{y_1}x_0,0)+\varphi([x_0,x_1],[y_0,y_1]).
\end{eqnarray*}
Notice that this equation is equivalent to the simpler $\varphi(\Lie_{y}x,0)=-\omega(y,\mu(x))$. \\
It is rather known that if $\delta \omega=0$, $\omega$ induces an (central) extension of $\hh$, 
\begin{eqnarray*}
\xymatrix{
0 \ar[r] & \Rr \ar[r]^{\bar{0}\times I\quad} & \hh\oplus ^\omega \Rr  \ar[r]^{\quad pr_1} & \hh \ar[r] & 0,
}
\end{eqnarray*}
where $\hh\oplus ^\omega\Rr$ is the semi-direct sum induced by $\omega$, that is the space $\hh\oplus\Rr$ endowed with the twisted bracket
\begin{eqnarray*}
[(y_0,\lambda_0),(y_1,\lambda_1)]_\omega=([y_0,y_1],-\omega(y_0,y_1)).
\end{eqnarray*}
\begin{lemma}
If $d(\omega,\varphi)=0$, then 
\begin{eqnarray*}
\xymatrix{
\mu _\varphi: \gg \ar[r] & \hh\oplus ^\omega\Rr : x \ar@{|->}[r] & (\mu(x),\varphi(x,0))
}
\end{eqnarray*}
defines a crossed module for the action $\Lie_{(y,\lambda)}x:=\Lie_y x$.
\end{lemma}
\begin{proof}
First, let us verify that $\mu _\varphi$ is indeed a Lie algebra homomorphism. It is clearly linear and
\begin{align*}
\mu _\varphi([x_0,x_1]) & = (\mu([x_0,x_1]),\varphi([x_0,x_1],0))        \\
                        & = ([\mu(x_0),\mu(x_1)],-\omega(\mu(x_0),\mu(x_1))) \\
                        & = [(\mu(x_0),\varphi(x_0,0)),(\mu(x_1),\varphi(x_1,0))]_\omega ,
\end{align*}
here $\partial \omega((0,x_0),(0,x_1))=\delta \varphi((0,x_0),(0,x_1))$ justifies the second equality. Now, the action is still a Lie algebra homomorphism and is still by derivations because $\lambda$ does not change the action
\begin{eqnarray*}
\Lie_{[(y_0,\lambda_0),(y_1,\lambda_1)]}x=\Lie_{[y_0,y_1]}x=[\Lie_{y_0},\Lie_{y_1}]x =[\Lie_{(y_0,\lambda_0)},\Lie_{(y_1,\lambda_1)}]x, 
\end{eqnarray*}
\begin{eqnarray*}
\Lie_{(y,\lambda)}[x_1,x_2]=\Lie_y[x_1,x_2]=[\Lie_y x_1,x_2]+[x_1,\Lie_y x_2]=[\Lie_{(y,\lambda)} x_1,x_2]+[x_1,\Lie_{(y,\lambda)} x_2].
\end{eqnarray*}
As for the equivariance of $\mu_\varphi$, on the one hand we have got
\begin{eqnarray*}
\mu _\varphi (\Lie_{(y,\lambda )}x)=(\mu(\Lie_y x),\varphi(\Lie_y x,0)),
\end{eqnarray*}
while on the other,
\begin{align*}
[(y,\lambda ),\mu_\varphi(x)]_\omega & = [(y,\lambda ),(\mu(x),\varphi(x,0))]_\omega \\
                                     & = ([y,\mu(x)],-\omega(y,\mu(x))).                                     
\end{align*}
The first entry coincides because $\mu$ is a crossed module morphism, whereas, as we pointed out, $\partial\omega+\delta\varphi=0$ is equivalent to
\begin{eqnarray*}
\varphi(\Lie_{y}x,0)=-\omega(y,\mu(x)).
\end{eqnarray*}
Finally, one sees that the infinitesimal Peiffer equation holds as,
\begin{align*}
\Lie_{\mu_\varphi(x_2)}x_1 & =\Lie_{(\mu(x_2),\varphi(x_2,0))}x_1 \\
                           & =\Lie_{\mu(x_2)}x_1 =[x_2,x_1];
\end{align*} 
thus proving the lemma.

\end{proof}
As a consequence of this lemma, we see that there is an induced short exact sequence of Lie $2$-algebras that we write using their associated crossed modules
\begin{eqnarray*}
\xymatrix{
0 \ar[r] & 0 \ar[d]\ar[r] & \gg \ar[d]^{\mu_\varphi}\ar[r]^{Id} & \gg \ar[d]^\mu \ar[r] & 0  \\
0 \ar[r] & \Rr \ar[r]     & \hh\oplus^\omega\Rr \ar[r]^{\quad pr_1}   & \hh \ar[r]            & 0.  
}
\end{eqnarray*}

\begin{prop}
Let $(\omega ,\varphi ),(\omega ',\varphi ')\in\Omega^2_{tot}(\gg_1 )$. If there exists a $\phi\in\Omega^1_{tot}(\gg _1)=\hh^*$, such that $(\omega,\varphi)-(\omega ',\varphi ')=d\phi$, then the induced extensions are isomorphic.
\end{prop}
\begin{proof}
First, recall that the object extensions are isomorphic via
\begin{eqnarray*}
\begin{pmatrix}
    I     & 0 \\
    -\phi & 1 
\end{pmatrix}\xymatrix{ 
:\hh\oplus^\omega\Rr \ar[r] & \hh\oplus^{\omega '}\Rr:(y,\lambda ) \ar@{|->}[r] & (y,\lambda-\phi (y)).
}
\end{eqnarray*}
We claim that this together with the identity of $\gg$ induce the isomorphism between the extensions. Indeed, using the notation from the previous lemma
\begin{align*}
\begin{pmatrix}
    I     & 0 \\
    -\phi & 1 
\end{pmatrix} \mu _\varphi(x) & =\begin{pmatrix}
    I     & 0 \\
    -\phi & 1 
\end{pmatrix} (\mu(x),\varphi(x,0)) \\
                                & =(\mu(x),\varphi(x,0)-\phi(\mu(x))) \\
                                & =(\mu(x),\varphi '(x,0)) = \mu _{\varphi '}(x),
\end{align*}
also, trivially $Id(\Lie_{(y,\lambda )}x)=\Lie_{(y,\lambda-\phi(y))}Id(x)$, thus finishing the proof.

\end{proof}
The fact that the second $2$-cohomology group classifies a certain type of extensions is a happy accident, because it suggests an extension of the classical theory of Lie algebra cohomology. Since in the classical theory extensions are classified by a cohomology with coefficients in a representation, we will devote the following subsection to study extensions of Lie $2$-algebras.

\subsection{Extensions of Lie 2-algebras and their associated cocycle equations}
We study abstract extensions to easily relate them to the spaces of $1$-cochains and $2$-cochains and the corresponding subspaces of $2$-cocycles and $2$-coboundaries. We look to express extensions in terms of simpler data and decide the equations these need to satisfy. \\
\begin{prop}\label{2-cocycles}
Let $\rho$ be a $2$-representation of $\xymatrix{\gg \ar[r]^\mu & \hh}$ on $\xymatrix{W \ar[r]^\phi & V}$. Given a triple $(\omega,\alpha,\varphi)\in((\hh^*\wedge\hh^*)\otimes V)\oplus(\hh^*\otimes\gg^*\otimes W)\oplus(\gg^*\otimes V)$, it defines a $2$-extension  \begin{eqnarray*}
\xymatrix{
0 \ar[r] & W \ar[d]_\phi\ar@{^{(}->}[r] & \gg {}_{\rho_0^1\circ\mu}\oplus^{\omega_1}W \ar[d]_\epsilon\ar[r]^{\qquad pr_1} & \gg \ar[d]^\mu\ar[r] & 0  \\
0 \ar[r] & V \ar@{^{(}->}[r]            & \hh {}_{\rho_0^0}\oplus^{\omega_1}V \ar[r]_{\qquad pr_1}                        & \hh \ar[r]           & 0, 
}
\end{eqnarray*}
with 
\begin{eqnarray*}
\omega_1(x_0,x_1)           & := & \rho_1(x_1)\varphi(x_0)+\alpha(\mu(x_0);x_1),  \\
\epsilon(x,w)                & = & (\mu(x),\phi(w)+\varphi(x)),                   \\
\Lie^{\epsilon}_{(y,v)}(x,w) & = & (\Lie _y x,\rho_0^1(y)w-\rho_1(x)v-\alpha(y;x))
\end{eqnarray*} 
if, and only if the following equations are satisfied 
\begin{itemize}
\item[i)] $\delta\omega =0$. Explicitely, for all triples $y_0,y_1,y_2\in\hh$, 
\begin{eqnarray*}
\rho^0_0(y_0)\omega(y_1,y_2)-\omega([y_0,y_1],y_2)+\circlearrowleft =0.
\end{eqnarray*}  
Here, $\circlearrowleft$ stands for cyclic permutations. 
\item[ii)] $\rho_1(x_0)\varphi(x_1)+\alpha(\mu(x_1);x_0)+\circlearrowleft =0$
\item[iii)] For all triples $x_0,x_1,x_2\in\gg$, 
\begin{eqnarray*}
\rho_0^1(\mu(x_0))(\rho_1(x_2)\varphi(x_1)+\alpha(\mu(x_2);x_1))-\rho_1(x_2)\varphi([x_0,x_1])-\alpha(\mu([x_0,x_1]);x_2)+\circlearrowleft =0.
\end{eqnarray*}
\item[iv)] $\omega(y,\mu(x))=\phi\circ\alpha(y;x)+\rho^0_0(y)\varphi(x)-\varphi(\Lie_y x)$
\item[v)] $\alpha([y_0,y_1];x)-\rho_1(x)\omega(y_0,y_1)=\rho^1_0(y_0)\alpha(y_1;x)-\alpha(y_1;\Lie_{y_0}x)+\circlearrowleft$
\item[vi)] For the contraction $a_y:=\alpha(y;-)\in\gg^*\otimes W$ seen as a $1$-cocycle with values in $\rho^1_0\circ\mu$,
\begin{eqnarray*}
\delta a_y(x_0,x_1)=\rho^1_0(y)\omega_1(x_0,x_1)-(\omega_1(\Lie _y x_0,x_1)-\omega_1(\Lie _y x_1,x_0))
\end{eqnarray*}
\end{itemize}
\end{prop}
There is nothing to these equations, in fact in most of the examples we have computed explicitly, they are redundant. In the proof, the reader might find the meaning of each of these equations.
\begin{proof}
We make the computations necessary to prove that a triple $(\omega ,\alpha ,\varphi)$ subject to the equations in the statement defines a $2$-extension. \\
First, the usual theory of Lie algebra extensions, tells us that item $i)$ says that $\omega$ is a $2$-cocycle with values in the representation $\rho_0^0$, and, as such, it defines a Lie bracket on $\hh\oplus V$. Analogously, $\omega_1$ defines an extension if, and only if $\omega_1$ is a $2$-cocycle with values in the representation $\rho_0^1\circ\mu$. With these, the equations in the statement have the following meaning
\begin{itemize}
\item[ii)] says $\omega_1$ is skew-symmetric. 
\item[iii)] says $\omega_1$ is a $2$-cocycle.
\item[iv)] evaluated at $y=\mu(x')$ says that $\epsilon$ is a Lie algebra homomorphism; indeed, 
\begin{align*}
\epsilon([(x',w'),(x,w)]) & =\epsilon([x',x],\rho_0^1(\mu(x'))w-\rho_0^1(\mu(x))w'-\omega_1(x',x)) \\
                          & =(\mu([x',x]),\phi(\rho_0^1(\mu(x'))w-\rho_0^1(\mu(x))w'-\omega_1(x',x))+\varphi([x',x])),
\end{align*}
while on the other hand,
\begin{align*}
[\epsilon(x',w'),\epsilon(x,w)] & =[(\mu(x'),\phi(w')+\varphi(x')),(\mu(x),\phi(w)+\varphi(x))] \\
                                & =([\mu(x'),\mu(x)],-\omega(\mu(x'),\mu(x))+ \\
                                & \qquad\qquad\rho_0^0(\mu(x'))(\phi(w)+\varphi(x))-\rho_0^0(\mu(x))(\phi(w')+\varphi(x'))).
\end{align*}
Clearly, the first entries coincide; moreover, since $\rho_0^0(\mu(x))=\phi\rho_1(x)$, $\epsilon$ is a Lie algebra homomorphism if, and only if
\begin{eqnarray*}
\varphi([x',x])-\phi\circ\omega_1(x',x)=\rho_0^0(\mu(x'))\varphi(x)-\rho_0^0(\mu(x))\varphi(x')-\omega(\mu(x'),\mu(x)).
\end{eqnarray*}
Replacing, the definition of $\omega_1$ yields
\begin{eqnarray*}
\varphi([x',x])-\phi\circ\alpha(\mu(x');x)=\rho_0^0(\mu(x'))\varphi(x)-\omega(\mu(x'),\mu(x)),
\end{eqnarray*}
which is precisely the equation in item $iv)$. Additionally, this equation also implies that $\epsilon$ is equivariant, as the following computation shows. 
\begin{align*}
\epsilon(\Lie^{\epsilon}_{(y,v)}(x,w)) & =\epsilon(\Lie _y x,\rho_0^1(y)w-\rho_1(x)v-\alpha(y;x)) \\
                                       & =(\mu(\Lie _y x),\phi(\rho_0^1(y)w-\rho_1(x)v-\alpha(y;x))+\varphi(\Lie _y x)),
\end{align*}
while on the other hand,
\begin{align*}
[(y,v),\epsilon(x,w)] & =[(y,v),(\mu(x),\phi(w)+\varphi(x))] \\
                      & =([y,\mu(x)],\rho_0^0(y)(\phi(w)+\varphi(x))-\rho_0^0(\mu(x))v-\omega(y,\mu(x))).
\end{align*}
Again, it is clear that the first entries coincide, and using again the relation $\rho_0^0(\mu(x))=\phi\rho_1(x)$ together with $\rho_0(y)\in\ggl (\phi)_0$, $\epsilon$ is equivariant if, and only if
\begin{eqnarray*}
\varphi(\Lie _y x)-\phi\circ\alpha(y;x)=\rho_0^0(y)\varphi(x)-\omega(y,\mu(x)).
\end{eqnarray*}
\item[v)] says $\Lie$ is a Lie algebra homomorphism and thus an action: 
\begin{align*}
\Lie^{\epsilon}_{[(y_0,v_0),(y_1,v_1)]}(x,w) & =\Lie^{\epsilon}_{([y_0,y_1],\rho_0^0(y_0)v_1-\rho_0^0(y_1)v_0-\omega(y_0,y_1))}(x,w) \\
                                             & =(\Lie _{[y_0,y_1]} x,\rho_0^1([y_0,y_1])w-\alpha([y_0,y_1];x) \\
                                             & \qquad\qquad-\rho_1(x)(\rho_0^0(y_0)v_1-\rho_0^0(y_1)v_0-\omega(y_0,y_1))). 
\end{align*}
On the other hand,
\begin{align*}
\Lie^{\epsilon}_{(y_0,v_0)}\Lie^{\epsilon}_{(y_1,v_1)}(x,w) 
                        & =\Lie^{\epsilon}_{(y_0,v_0)}(\Lie _{y_1} x,\rho_0^1(y_1)w-\rho_1(x)v_1-\alpha(y_1;x)) \\
                        & =(\Lie _{y_0} \Lie _{y_1} x,\rho_0^1(y_0)(\rho_0^1(y_1)w-\rho_1(x)v_1-\alpha(y_1;x))  \\
                        & \qquad\qquad-\rho_1(\Lie _{y_1} x)v_0-\alpha(y_0;\Lie _{y_1} x)).
\end{align*}
Then, considering the cyclic difference, one realizes that the first entries coincide. Using the fact that $\rho_0^1$ is a Lie algebra representation and the fact that $\rho_1$ respects is compatible with the actions, one is left with the relation
\begin{eqnarray*}
\rho_1(x)\omega(y_0,y_1)-\alpha([y_0,y_1];x)=\rho_0^1(y_1)\alpha(y_0;x)+\alpha(y_1;\Lie _{y_0} x))-\circlearrowleft.
\end{eqnarray*}
\item[vi)] says that $\Lie$ acts by derivations:
\begin{align*}
\Lie^{\epsilon}_{(y,v)}[(x_0,w_0),(x_1,w_1)] & =\Lie^{\epsilon}_{(y,v)}([x_0,x_1],\rho_0^1(\mu(x_0))w_1-\rho_0^1(\mu(x_1))w_0-\omega_1(x_0,x_1)) \\
                                             & =(\Lie _{y} [x_0,x_1],\rho_0^1(y)(\rho_0^1(\mu(x_0))w_1-\rho_0^1(\mu(x_1))w_0-\omega_1(x_0,x_1)) \\
                                             & \qquad\qquad -\rho_1([x_0,x_1])v-\alpha(y;[x_0,x_1])) 
\end{align*}
On the other hand,
\begin{align*}
[\Lie^{\epsilon}_{(y,v)}(x_0,w_0),(x_1,w_1)] 
                        & =[(\Lie _{y} x_0,\rho_0^1(y)w_0-\rho_1(x_0)v-\alpha(y;x_0)),(x_1,w_1)] \\
                        & =([\Lie _{y} x_0,x_1],\rho_0^1(\mu(\Lie _{y} x_0))w_1-\omega_1(\Lie _{y} x_0,x_1) \\
                        & \qquad\qquad -\rho_0^1(\mu(x_1))(\rho_0^1(y)w_0-\rho_1(x_0)v-\alpha(y;x_0))).
\end{align*}
Considering again the cyclic difference, one realizes that the first entries coincide. Since $\mu$ is equivariant and $\rho_0^1$ is a Lie algebra representation, we are left with the relation
\begin{eqnarray*}
-\rho_0^1(y)\omega_1(x_0,x_1)-\alpha(y;[x_0,x_1])=-\omega_1(\Lie _{y} x_0,x_1)+\rho_0^1(\mu(x_1))\alpha(y;x_0))-\circlearrowleft.
\end{eqnarray*}
\end{itemize}
Finally, the infinitesimal Peiffer equation holds by the very definition of $\omega_1$; indeed,
\begin{align*}
[(x_0,w_0),(x_1,w_1)] & =([x_0,x_1],\rho_0^1(\mu(x_0))w_1-\rho_0^1(\mu(x_1))w_0-\omega_1(x_0,x_1)),
\end{align*}
whereas
\begin{align*}
\Lie^{\epsilon}_{\epsilon(x_0,w_0)}(x_1,w_1) & =\Lie^{\epsilon}_{(\mu(x_0),\phi(w_0)+\varphi(x_0))}(x_1,w_1) \\
                                             & =(\Lie _{\mu(x_0)}x_1,\rho_0^1(\mu(x_0))w_1-\rho_1(x_1)(\phi(w_0)+\varphi(x_0))-\alpha(\mu(x_0),x_1)).
\end{align*}
Due to the relation $\rho_0^1(\mu(x))=\rho_1(x)\phi$, we are left precisely with the defining equation for $\omega_1$.

\end{proof}

Next, we analyze what happens when we have got equivalent extensions as in
\begin{eqnarray*}
\xymatrix{
         &                     & \mathfrak{e}_1 \ar[dd]\ar@{.>}[ddr]^{\psi_1}\ar[drr] &          &                    &    \\
0 \ar[r] & W \ar[dd]\ar[ur]\ar[drr] &                                 &                          & \gg \ar[r] \ar[dd] & 0  \\
         &                     & \mathfrak{e}_0 \ar[drr]\ar@{.>}[ddr]_{\psi_0} & \mathfrak{f}_1 \ar[dd]\ar[ur] &      &    \\
0 \ar[r] & V \ar[ur]\ar[drr]        &                                 &                          & \hh \ar[r]         & 0. \\
         &                          &                                          & \mathfrak{f}_0 \ar[ur]        &      &   
}
\end{eqnarray*}
Picking a splitting of either extension and composing it with the isomorphism, one gets a splitting for the other extension. In picking the splittings compatibly so, the induced $2$-representations are identical. We use these compatible splittings to identify both $\mathfrak{e}$ and $\mathfrak{f}$ with their respective semi-direct sums, and we write $\psi$ in these coordinates. Since both components of $\psi$ are linear, and respect both inclusions and projections
\begin{eqnarray}\label{anIsoOfExts}
\psi_k(z,a)=(z,a+\lambda_k(z))
\end{eqnarray}
for some linear maps $\xymatrix{\lambda_0:\hh \ar[r] & V}$ and $\xymatrix{\lambda_1:\gg \ar[r] & W}$.
\begin{prop}\label{2-coboundaries}
Let $\rho$ be a $2$-representation of $\xymatrix{\gg \ar[r]^\mu & \hh}$ on $\xymatrix{W \ar[r]^\phi & V}$. Given two $2$-cocycles  $(\omega_k,\alpha_k,\varphi_k)$ as in the previous proposition, the induced extensions are equivalent if, and only if there are linear maps $\xymatrix{\lambda_0:\hh \ar[r] & V}$ and $\xymatrix{\lambda_1:\gg \ar[r] & W}$ verifying  
\begin{itemize}
\item $\omega_2-\omega_1 =\delta\lambda_0$. Explicitly, for all triples $y_0,y_1\in\hh$, 
\begin{eqnarray*}
\omega_2(y_0,y_1)-\omega_1(y_0,y_1) =\rho^0_0(y_0)\lambda_0(y_1)-\rho^0_0(y_1)\lambda_0(y_0)-\lambda_0([y_0,y_1])
\end{eqnarray*}  
\item $\alpha_2(y;x)-\alpha_1(y;x)=\rho^1_0(y)(\lambda_1(x))-\lambda_1(\Lie _y x)-\rho_1(x)(\lambda_0(y))$
\item $\varphi_2(x)-\varphi_1(x)=\lambda_0(\mu(x))-\phi(\lambda_1(x))$
\end{itemize}
\end{prop}
\begin{proof}
The first equation is the classic identification of isomorphic extensions with cohomologous cocycles. \\
The second equation says that if the isomorphim of extensions $\psi$ is defined by the formula \ref{anIsoOfExts}, it respects the actions. Indeed, 
\begin{align*}
    \psi_1(\Lie_{(y,v)}(x,w)) & =\psi_1(\Lie _y x,\rho^1_0(y)w -\rho_1(x)v -\alpha_1(y;x)) \\
                              & =(\Lie_yx,\rho^1_0(y)w-\rho_1(x)v-\alpha_1(y;x)+\lambda_1(\Lie_yx)),
\end{align*}
whereas,
\begin{align*}
    \Lie_{\psi_0(y,v)}\psi_1(x,w) & =\Lie_{(y,v+\lambda_0(y))}(x,w+\lambda_1(x)) \\
                                  & =(\Lie_yx,\rho^1_0(y)(w+\lambda_1(x))-\rho_1(x)(v+\lambda_0(y))-\alpha_2(y;x)) .
\end{align*}
Thus, these expressions coincide if, and only if the second equation from the statement holds.
The third equation says that $\psi$ commutes with the structural maps of the crossed modules. Indeed,
\begin{align*}
 \psi_0(\epsilon_1(x,w)) & =\psi_0(\mu(x),\phi(w)+\varphi_1(x)) \\
                         & =(\mu(x),\phi(w)+\varphi_1(x)+\lambda_0(\mu(x))),
\end{align*}
while on the other hand,
\begin{align*}
 \epsilon_2(\psi_1(x,w)) & =\epsilon_2(x,w+\lambda_1(x)) \\
                         & =(\mu(x),\phi(w+\lambda_1(x))+\varphi_2(x)).
\end{align*}
Thus, these expressions coincide if, and only if the last equation of the statement holds. To conclude the proof, notice that if the equations in the statement are verified, the cocycles
\begin{eqnarray*}
\omega_k'(x_0,x_1)=\rho_1(x_1)\varphi_k(x_0)+\alpha_k(\mu(x_0);x_1),
\end{eqnarray*}
for $k\in\lbrace 1,2\rbrace$ defining the top extensions are cohomologous too. Indeed,
\begin{align*}
    (\omega_2'-\omega_1')(x_0,x_1) & =\rho_1(x_1)(\varphi_2-\varphi_1)(x_0)+(\alpha_2-\alpha_1)(\mu(x_0);x_1) \\
                                   & =\rho_1(x_1)(\lambda_0(\mu(x_0))-\phi(\lambda_1(x_0)))+\rho^1_0(\mu(x_0))(\lambda_1(x_1))+ \\
                                   & \qquad\qquad\qquad\qquad\qquad\qquad -\lambda_1(\Lie_{\mu(x_0)}x_1)-\rho_1(x_1)(\lambda_0(\mu(x_0)) \\
                                   & =-\rho_1(x_1)\phi(\lambda_1(x_0))+\rho^1_0(\mu(x_0))(\lambda_1(x_1))-\lambda_1([x_0,x_1])=\delta\lambda_1(x_0,x_1),
\end{align*}
where the last equality holds, because $\rho$ is a $2$-representation, and therefore, $\rho_1(x)\phi=\rho_0^1(\mu(x))$.

\end{proof}


We conclude this section by studying the cohomology of the complex from theorem \ref{The2AlgCx}. \\
Start out with an element $v\in V=C^0(\gg_1,\phi)=C^{0,0}_0(\gg_1,\phi)$, then its differential is given by
\begin{eqnarray*}
\nabla v=(\delta v,\delta_{(1)}v,\cancelto{0}{\partial v})\in C^1(\gg_1,\phi)
\end{eqnarray*}
If $v$ is a $0$-cocycle, then for all $y\in\hh$ and for all $x\in\gg$,
\begin{eqnarray*}
\rho_0^0(y)v=0 & \textnormal{and} & \rho_1(x)v=0;
\end{eqnarray*}
therefore,
\begin{eqnarray*}
H^0_\nabla(\gg_1,\phi)=V^{\gg_1}:=\lbrace v\in V:\bar{\rho}_{(x,y)}(0,v)=0,\quad\forall (x,y)\in\gg\oplus_\Lie\hh\rbrace ,
\end{eqnarray*}
where $\bar{\rho}$ is the honest representation of proposition \ref{honestAlgRep}. \\
A $1$-cochain $\lambda$ is a triple $(\lambda_0,\lambda_1,v)\in(\hh^*\otimes V)\oplus(\gg^*\otimes W)\oplus V$, and its differential has six entries that we write using their coordinates, i.e. $\nabla\lambda^{p,q}_r\in C^{p,q}_r(\gg _1,\phi)$
\begin{align*}
\nabla\lambda^{0,2}_0 & =\delta\lambda_0                             &                       & & &                               \\
\nabla\lambda^{1,1}_0 & =-\partial\lambda_0+\delta v-\Delta\lambda_1 & \nabla\lambda^{0,1}_1 & =-\delta_{(1)}\lambda_0+\delta'\lambda_1 & & \\
\nabla\lambda^{2,0}_0 & =\partial v=v                                & \nabla\lambda^{1,0}_1 & =\delta_{(1)}v-\partial\lambda _1=\delta_{(1)}v & \nabla\lambda^{0,0}_2 & =\delta_{(1)}\lambda_1 
\end{align*}
Schematically,
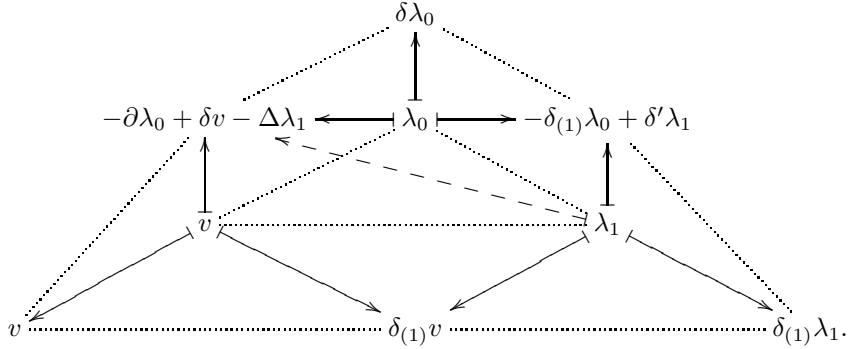
\begin{eqnarray*}
\xymatrix{
  & & \delta\lambda_0 \ar@{.}[dl]\ar@{.}[dr] & &  \\
  & -\partial\lambda_0+\delta v-\Delta\lambda_1 & \lambda_0 \ar@{|->}[u]\ar@{|->}[r]\ar@{|->}[l]\ar@{.}[dr] & -\delta_{(1)}\lambda_0+\delta'\lambda_1 & \\
  & v \ar@{|->}[dl]\ar@{|->}[u]\ar@{|->}[dr]\ar@{.}[ur]\ar@{.}[rr] &  & \lambda_1 \ar@{|->}[dr]\ar@{|->}[u]\ar@{|->}[dl]\ar@{|-->}[ull] & \\
v \ar@{.}[uur] \ar@{.}[rr] & & \delta_{(1)}v \ar@{.}[rr] & & \delta_{(1)}\lambda_1 .\ar@{.}[uul]
}
\end{eqnarray*}
Here the solid arrows represent the differentials, the dashed arrow represents the difference map, and the pointed polygons represent elements of the same degree. If $\lambda$ is a $1$-cocylce, then $v=0$, $(\lambda_0,\lambda_1)\in Der(\hh,V)\oplus Der(\gg,W)$ and the following relations hold for all pairs $(y,x)\in\hh\times\gg$:
\begin{eqnarray*}
\phi(\lambda_1(x))=\lambda_0(\mu(x)), \\
\lambda_1(\Lie_y x)=\rho_1(x)\lambda_0(y)-\rho_0^1(y)\lambda_1(x).
\end{eqnarray*}
The first of these relations says that 
\begin{eqnarray*}
\xymatrix{
\bar{\lambda}:\gg\oplus\hh \ar[r] & W\oplus V:(x,y) \ar@{|->}[r] & (\lambda_1(x),\lambda_0(y))
}
\end{eqnarray*}
is a map of $2$-vector spaces. The second, says that $\bar{\lambda}\in Der(\gg\oplus_\Lie\hh,W\oplus V)$ with respect to $\bar{\rho}$. Notice that since $\nabla v$ can also be seen as $\bar{\rho}_{(x,y)}(0,v)$; by analogy, we can define:
\begin{Def}
The \textit{space of derivations} of a Lie $2$-algebra $\gg_1$ with respect to a $2$-representation $\rho$ on the $2$-vector $\mathbb{V}=\xymatrix{W \ar[r]^\phi & V}$ is defined to be
\begin{eqnarray*}
Der(\gg_1,\phi):=\lbrace (\lambda_1,\lambda_0)\in Hom_{2-Vect}(\gg_1,\mathbb{V}):(\lambda_1,\lambda_0)\in Der(\gg\oplus_\Lie\hh,W\oplus V)\rbrace .
\end{eqnarray*}
The space of \textit{inner derivations}, on the other hand, is defined to be
\begin{eqnarray*}
Inn(\gg_1,\phi):=\lbrace (\lambda_1,\lambda_0)\in Der(\gg_1,\phi):(\lambda_1(x),\lambda_0(y))=\bar{\rho}_{(x,y)}(0,v)\textnormal{ for some }v\in V\rbrace .
\end{eqnarray*}
\end{Def}
With these definitions, we can write
\begin{eqnarray*}
H^1_\nabla(\gg_1,\phi)=Out(\gg_1,\phi):=Der(\gg_1,\phi)/Inn(\gg_1,\phi).
\end{eqnarray*}
A $2$-cochain $\vec{\omega}$ is a $6$-tuple $(\omega_0,\alpha,\varphi,\omega_1,\lambda,v)$ where
\begin{align*}
\omega_0 & \in\bigwedge^2\hh^*\otimes V &          &                               &          &   \\
\varphi  & \in\gg_1^*\otimes V          & \alpha   & \in\hh^*\otimes\gg^*\otimes W &          &   \\
v        & \in V                        & \lambda  & \in\gg^*\otimes W             & \omega_1 & \in\bigwedge^2\gg^*\otimes W . 
\end{align*} 
Let us compute the differential using coordinates as before. First, $\nabla\vec{\omega}^{3,0}_0=\partial v=0$; for the other values,
\begin{align*}
\nabla\vec{\omega}^{0,3}_0 & =\delta\omega_0 & & & &  \\
\nabla\vec{\omega}^{1,2}_0 & =\partial\omega_0+\delta\varphi-\Delta\alpha+\Delta_2\omega_1 & \nabla\vec{\omega}^{0,2}_1 & =\delta_{(1)}\omega_0+\delta'\alpha & & \\
\nabla\vec{\omega}^{2,1}_0 & =-\partial\varphi+\delta v-\Delta\lambda & \nabla\vec{\omega}^{1,1}_1 & =-\delta_{(1)}\varphi+\partial\alpha+\delta'\lambda+\Delta\omega_1 & \nabla\vec{\omega}^{0,1}_2 & =-\delta_{(1)}\alpha+\delta^{(2)}\omega_1 \\
\nabla\vec{\omega}^{2,0}_1 & =\delta_{(1)}v-\cancelto{\lambda}{\partial\lambda} & \nabla\vec{\omega}^{1,0}_2 & =\cancelto{0}{\partial\omega_1}+\delta_{(1)}\lambda & \nabla\vec{\omega}^{0,0}_3 & =\delta_{(1)}\omega_1
\end{align*}
If we assume  $v=0$, $\lambda$ vanishes too and these equations correspond to those in the statement of proposition \ref{2-cocycles}. First, $\partial\varphi(x_0,x_1,y)=\varphi(0,y+\mu(x_1))=0$ says that $\varphi$ does not depend on $\hh$, as was supposed in the statement. Then, the equation $\nabla\vec{\omega}^{1,1}_1=0$ coincides with the definition of $\omega_1$ in the statement. Finally, for the numbering of the equations in the statement of proposition \ref{2-cocycles}:
\begin{itemize}
\item Equation $i)$ is literally $\delta\omega_0 =0$.  
\item Equation $ii)$ just said that $\omega_1$ was skew-symmetric, so we've got it already, as our $\omega_1$ is alternating by definition.
\item Equation $iii)$ said that $\omega_1$ was a cocycle, so it follows from $\nabla\vec{\omega}^{0,0}_3=0$.
\item Equation $iv)$ is equivalent to $\nabla\vec{\omega}^{1,2}_0=0$ evaluated at $\begin{pmatrix}
0 & x \\
y & 0
\end{pmatrix}\in\gg_1^2$. 
\item Equation $v)$ is exactly $\nabla\vec{\omega}^{0,2}_1=0$.
\item Equation $vi)$ is exactly $\nabla\vec{\omega}^{0,1}_2=0$.
\end{itemize}
Moreover, for a pair of $2$-cocycles 
\begin{eqnarray*}
(\vec{\omega})^k=(\omega_0^k,\alpha^k,\varphi^k,\omega_1^k,\lambda^k,v^k), & k\in\lbrace 1,2\rbrace
\end{eqnarray*}
with $v^k=0$, if they are cohomologous we recover the equations in proposition \ref{2-coboundaries}. Indeed, if
\begin{eqnarray*}
(\vec{\omega})^2-(\vec{\omega})^1=\nabla(\lambda_0,\lambda_1,v),
\end{eqnarray*}
coordinate-wise we've got
\begin{align*}
\omega_0^2-\omega_0^1 & =\delta\lambda_0                             &                   & & &                               \\
\varphi^2-\varphi^1   & =-\partial\lambda_0+\delta v-\Delta\lambda_1 & \alpha^2-\alpha^1 & =-\delta_{(1)}\lambda_0+\delta'\lambda_1 & & \\
0                     & =v                                           & 0                 & =\delta_{(1)}v                           & \omega_1^2-\omega_1^1 &  =\delta_{(1)}\lambda_1 ;
\end{align*}
thus, explicitely 
\begin{align*}
    (\varphi^2-\varphi^1)(x) & =\lambda_0(\mu(x))-\phi(\lambda_1(x)) \\
    (\alpha^2-\alpha^1)(y;x) & =-\rho_1(x)\lambda_0(y)+\rho_0^1(y)\lambda_1(x)-\lambda_1(\Lie_y x)
\end{align*}
which coincide with the referred equations, as we claimed. \\
In the end, this is the theorem this theory is dedicated to. 
\begin{theorem}\label{H2Alg}
$H^2_\nabla(\gg_1,\phi)$ classifies $2$-extensions.
\end{theorem}

\bibliographystyle{plain}
\bibliography{biblos.bib} 		             


\end{document}